\documentclass[11pt,leqno]{article}
\usepackage{amsmath, amscd, amsthm, amssymb, graphics, xypic, mathrsfs, setspace, fancyhdr, times, enumerate}
\entrymodifiers={+!!<0pt,\fontdimen22\textfont2>}
\DeclareFontFamily{OT1}{pzc}{}
\DeclareFontShape{OT1}{pzc}{m}{it}%
             {<-> s * [1.195] pzcmi7t}{}
\DeclareMathAlphabet{\mathscr}{OT1}{pzc}%
                                 {m}{it}
\usepackage[colorlinks=true,pagebackref=true]{hyperref} 
\hypersetup{backref}

\newcommand{\tensor}{\otimes}
\newcommand{\colim}{\operatorname{colim}}
\newcommand{\Spec}{\operatorname{Spec}}

\newcommand{\isomt}{{\stackrel{{\scriptscriptstyle{\sim}}}{\;\rightarrow\;}}}

\newcommand{\sma}{{\scriptstyle{\wedge}}}
\newcommand{\coker}{\operatorname{coker}}

\newcommand{\longtwoheadrightarrow}{\relbar\joinrel\twoheadrightarrow}

\renewcommand{\O}{{\mathcal O}}
\renewcommand{\hom}{\operatorname{Hom}}

\newcommand{\real}{{\mathbb R}}

\newcommand{\cplx}{{\mathbb C}}

\newcommand{\Z}{{\mathbb Z}}
\newcommand{\N}{{\mathbb N}}
\newcommand{\A}{{\mathbb A}}

\newcommand{\aone}{{\mathbb A}^1}
\newcommand{\pone}{{\mathbb P}^1}

\newcommand{\gm}[1]{{{\mathbf G}_{m}^{#1}}}
\renewcommand{\L}{{\mathcal L}}

\newcommand{\ho}[1]{\mathscr{H}({#1})}
\newcommand{\hop}[1]{\mathscr{H}_{\bullet}({#1})}

\newcommand{\bpi}{\boldsymbol{\pi}}
\newcommand{\piaone}{{\bpi}^{\aone}}

\newcommand{\Nis}{\operatorname{Nis}}

\newcommand{\Shv}{{\mathscr{Shv}}}
\newcommand{\Sm}{\mathscr{Sm}}

\newcommand{\Spc}{\mathscr{Spc}}

\newcommand{\K}{{{\mathbf K}}}

\newcommand{\hsnis}{\mathscr{H}_s^{\Nis}(k)}
\newcommand{\hspnis}{\mathscr{H}_{s,\bullet}^{\Nis}(k)}

\newcommand{\F}{{\mathcal F}}

\newcounter{intro}
\theoremstyle{plain}
\newtheorem{thm}{Theorem}[subsection]

\newtheorem{lem}[thm]{Lemma}
\newtheorem{cor}[thm]{Corollary}
\newtheorem{prop}[thm]{Proposition}
\newtheorem*{claim*}{Claim}  

\newtheorem*{thm*}{Theorem}
\newtheorem*{problem*}{Problem}

\newtheorem{thmintro}{Theorem}

\newtheorem{conjintro}[thmintro]{Conjecture}

\theoremstyle{definition}
\newtheorem{defn}[thm]{Definition}

\newtheorem{notation}[thm]{Notation}

\theoremstyle{remark}
\newtheorem{rem}[thm]{Remark}
\newtheorem{remintro}[thmintro]{Remark}

\newtheorem{ex}[thm]{Example}

\numberwithin{equation}{subsection}

\begin{document}
\pagestyle{fancy}
\renewcommand{\sectionmark}[1]{\markright{\thesection\ #1}}
\fancyhead{}
\fancyhead[LO,R]{\bfseries\footnotesize\thepage}
\fancyhead[LE]{\bfseries\footnotesize\rightmark}
\fancyhead[RO]{\bfseries\footnotesize\rightmark}
\chead[]{}
\cfoot[]{}
\setlength{\headheight}{1cm}

\author{\begin{small}Aravind Asok\thanks{Aravind Asok was partially supported by National Science Foundation Awards DMS-0900813 and DMS-0966589.}\end{small} \\ \begin{footnotesize}Department of Mathematics\end{footnotesize} \\ \begin{footnotesize}University of Southern California\end{footnotesize} \\ \begin{footnotesize}Los Angeles, CA 90089-2532 \end{footnotesize} \\ \begin{footnotesize}\url{asok@usc.edu}\end{footnotesize} \and \begin{small}Jean Fasel\thanks{Jean Fasel was supported by the DFG Grant SFB Transregio 45.}\end{small} \\ \begin{footnotesize}Fakult\"at Mathematik \end{footnotesize} \\ \begin{footnotesize}Universit\"at Duisburg-Essen, Campus Essen\end{footnotesize} \\ \begin{footnotesize}Thea-Leymann-Strasse 9, D-45127 Essen\end{footnotesize} \\ \begin{footnotesize}\url{jean.fasel@gmail.com}\end{footnotesize}}

\title{{\bf Splitting vector bundles outside the stable range and \\ $\aone$-homotopy sheaves of punctured affine spaces}}
\date{}
\maketitle

\begin{abstract}
We discuss the relationship between the ${\mathbb A}^1$-homotopy sheaves of ${\mathbb A}^n \setminus 0$ and the problem of splitting off a trivial rank $1$ summand from a rank $n$ vector bundle.  We begin by computing $\bpi_3^{\aone}({\mathbb A}^3 \setminus 0)$, and providing a host of related computations of ``non-stable" $\aone$-homotopy sheaves.  We then use our computation to deduce that a rank $3$ vector bundle on a smooth affine $4$-fold over an algebraically closed field having characteristic unequal to $2$ splits off a trivial rank $1$ summand if and only if its third Chern class (in Chow theory) is trivial.  This result provides a positive answer to a case of a conjecture of M.P. Murthy.
\end{abstract}
\setcounter{tocdepth}{1}
\begin{footnotesize}
\tableofcontents
\end{footnotesize}

\section{Introduction}
This paper is motivated in part by the following classical question: if $X$ is a smooth affine variety of dimension $d$ over a field $k$, under what conditions does a rank $r$ vector bundle on $X$ split as the direct sum of a rank $r-1$ vector bundle and a trivial bundle of rank $1$ (briefly: when does a rank $r$ vector bundle split off a trivial rank $1$ bundle)?  The main idea of this paper, which is the third in a series after \cite{AsokFaselSpheres} and \cite{AsokFaselThreefolds}, is to apply $\aone$-homotopy theory to provide some new results regarding this splitting problem.

The answer to the question posed in the previous paragraph depends on the relationship between $r$ and $d$.  For example, in 1958, Serre proved that if $X$ is an affine scheme of Krull dimension $d$, then any vector bundle on $X$ of rank $r > d$ is the direct sum of a vector bundle of rank $d$ and a trivial bundle \cite[Th\'eor\`eme 1]{Serre58}.  If $X$ is an affine curve, then answering the above question when $r = d$ is trivial (in fact, cancellation holds, i.e., stably isomorphic vector bundles are isomorphic). Answering the splitting question for  $r = d \geq 2$ led to a torrent of work.


It follows from  \cite[Theorem 1]{MurthySwan} that if $X$ is a smooth affine surface over an algebraically closed field, then a rank $2$ bundle on $X$ splits off a trivial rank $1$ summand if and only if its second Chern class in $CH^2(X)$ is zero.  In \cite[Corollary 2.4]{KumarMurthy} Murthy and Mohan Kumar showed that if $X$ is a smooth affine threefold over an algebraically closed field, then a rank $3$ bundle $\mathcal{E}$ splits off a trivial rank $1$ summand if and only if $0 = c_3(\mathcal{E}) \in CH^3(X)$.  Subsequently, Murthy \cite[Remark 3.6 and Theorem 3.7]{Murthy94} showed (in particular) that if $X$ is a smooth affine variety of dimension $d$ over an algebraically closed field $k$, and if $\mathcal{E}$ is a rank $d$ bundle on $X$, then $\mathcal{E}$ splits off a trivial rank $1$ summand if and only if $0 = c_d(\mathcal{E}) \in CH^d(X)$.

In \cite[Chapter 8]{MField}, Morel revisited the splitting problem in the setting of $\aone$-homotopy theory.  Using his classification theorem for vector bundles on smooth affine schemes (\cite[Theorem 8.1]{MField}), and techniques of obstruction theory, he was able to recast the splitting problem in terms precisely analogous to the classical topological theory of the Euler class (see, e.g., \cite{MilnorStasheff}).  In particular, he showed that, for smooth varieties of dimension $d$ over an arbitrary perfect field $k$, there is an ``Euler class" obstruction to splitting off a trivial rank $1$ bundle from a rank $r$ bundle with trivial determinant \cite[Theorem 8.14]{MField}; see also \cite[Corollaire 15.3.12]{FaselChowWitt} and \cite[Theorem 38]{FaselSrinivas} for $d=2,3$ (Morel also remarks that for vector bundles with non-trivial determinant \cite[Remark 8.15.2]{MField} it is possible to define a twisted Euler class that obstructs splitting off a trivial rank $1$ summand).  Morel's Euler class is the unique obstruction to splitting off a trivial rank $1$ summand if $r = d$, and if, furthermore, $k$ is algebraically closed, then the Euler class coincides with the top Chern class of the bundle.

Rank $d$ vector bundles on smooth affine $d$-folds are ``at the edge of the stable range".  Indeed, if $k$ is field of $2$-cohomological dimension $\leq 1$, e.g., if $k$ is algebraically closed, it follows from Bhatwadekar's extension \cite{BhatwadekarCancellation} of Suslin's cancellation theorem \cite{Suslin77} that rank $d$ vector bundles are determined by their class in $K_0(X)$.  {\em A priori}, one might not expect to be able to make any reasonable statements about the structure of vector bundles of rank $r<d$. Nevertheless, Murthy wrote that he did not know an example of a vector bundle $\mathcal{E}$ of rank $d-1$ on a smooth affine $d$-fold over an algebraically closed field $k$ such that $c_{d-1}(\mathcal{E}) = 0 \in CH^{d-1}(X)$ that does not split off a trivial rank $1$ summand \cite[p. 173]{Murthy97}.  Following a long established tradition, we reformulate the question implicit in Murthy's statement as a conjecture.

\begin{conjintro}[Murthy's splitting conjecture]
\label{conj:murthy}
If $X$ is a smooth affine $d$-fold over an algebraically closed field $k$ and $\mathcal E$ is a vector bundle of rank $d-1$ over $X$, then ${\mathcal E}$ splits off a trivial rank $1$ bundle if and only if $c_{d-1}(\mathcal E)=0$ in $CH^{d-1}(X)$.
\end{conjintro}

With one exception, we were not aware of any general (algebro-geometric) results regarding splitting vector bundles outside the stable range.  In \cite[Theorem 6.6]{AsokFaselThreefolds} we proved that, given an algebraically closed field $k$ and a smooth affine threefold over $k$, there is a unique rank $2$ vector bundle on $X$ with given $c_1$ and $c_2$; consequently, a rank $2$ vector bundle on such a variety splits off a trivial rank $1$ summand if and only if $c_2$ is trivial.  In particular, Conjecture \ref{conj:murthy} is true when $d=3$ under the additional assumption that $k$ has characteristic unequal to $2$.  In this work, we provide a solution to Conjecture \ref{conj:murthy} in case $d=4$, again under the additional assumption that $k$ has characteristic unequal to $2$.

\begin{thmintro}[See Subsections \ref{ss:primaryobstruction} and \ref{ss:secondaryobstruction}]
\label{thmintro:murthysplittingdimension4}
If $X$ is a smooth affine $4$-fold over an algebraically closed field $k$ having characteristic unequal to $2$ and if $\mathcal{E}$ is a rank $3$ vector bundle on $X$, then $\mathcal{E}$ splits off a trivial rank $1$ bundle if and only if $0 = c_3(\mathcal{E}) \in CH^3(X)$.
\end{thmintro}

This result and the one mentioned in the previous paragraph were deduced by studying the link between the splitting problem and $\aone$-homotopy theory.  To explain this relation, write $BGL_n$ for the classifying space for $GL_n$-torsors (the reader is encouraged to think of an appropriate infinite Grassmannian).  Write $\ho{k}$ for the Morel-Voevodsky $\aone$-homotopy category.  Given any smooth scheme $X$, we write $[X,BGL_n]_{\aone}$ for the set of morphisms in $\ho{k}$ from $X$ to $BGL_n$.  Morel showed in \cite[Theorem 8.1]{MField} that the set $[X,BGL_n]_{\aone}$ is canonically in bijection with the set $\mathscr{V}_n(X)$ of isomorphism classes of rank $n$ vector bundles on $X$ (provided $X$ is affine).

There is a morphism $BGL_{n-1} \to BGL_n$ corresponding to the inclusion map $GL_{n-1} \to GL_n$ sending an invertible matrix $M$ to the block diagonal $n \times n$-matrix with blocks $M$ and $1$.  This morphism induces a map $[X,BGL_{n-1}]_{\aone} \to [X,BGL_{n}]_{\aone}$ that sends the isomorphism class of a rank $n-1$ vector bundle $\mathcal{E}$ to the isomorphism class of the rank $n$ vector bundle $\mathcal{E} \oplus \O_X$.  Therefore, the splitting problem is equivalent to the following lifting question: given an element of $[X,BGL_n]_{\aone}$, when can it be lifted along the morphism $BGL_{n-1} \to BGL_n$ to an element of $[X,BGL_{n-1}]_{\aone}$?  Standard obstruction theoretic techniques from algebraic topology, adapted to the setting of $\aone$-homotopy theory, show that the obstructions to existence of such a lift are governed by the structure of the ($\aone$-)homotopy fiber of the above map $BGL_{n-1} \to BGL_n$. Morel explicitly identified this $\aone$-homotopy fiber by proving the existence of an $\aone$-fiber sequence \cite[Proposition 8.11]{MField} of the form:
\[
GL_n/GL_{n-1} \longrightarrow BGL_{n-1} \longrightarrow BGL_n.
\]
The left-hand morphism is the classifying map of the $GL_{n-1}$-torsor $GL_n\to GL_{n}/GL_{n-1}$. Moreover, the space $GL_n/GL_{n-1}$ is $\aone$-weakly equivalent to ${\mathbb A}^n\setminus 0$ which is a sphere in $\ho k$.  Alternatively, $BGL_{n-1}$ can be identified, up to $\aone$-weak equivalence, with the complement of the zero section of the tautological vector bundle over $BGL_n$, and the left hand map can be thought of, again up to $\aone$-weak equivalence, as the inclusion of the fiber over the base-point.

By obstruction theory, understanding the lifting question is then tantamount to understanding the (unstable) $\aone$-homotopy theory of ${\mathbb A}^n \setminus 0$.  To provide a positive answer to Murthy's question for a given integer $d$, the above approach requires as input sufficiently detailed information about the $d-1$st $\aone$-homotopy sheaf of ${\mathbb A}^{d-1} \setminus 0$.  In particular, in \cite[Theorem 3]{AsokFaselThreefolds} we computed $\bpi_2^{\aone}({\mathbb A}^2 \setminus 0)$.  In this paper, we deduce Theorem \ref{thmintro:murthysplittingdimension4} from a computation of $\bpi_3^{\aone}({\mathbb A}^3 \setminus 0)$.

\begin{thmintro}[See Theorem \ref{thm:pi3a3minus0} and Lemma \ref{lem:contractedT5}]
\label{thmintro:pi3a3minus0}
If $k$ is an infinite perfect field having characteristic unequal to $2$, there is a short exact sequence of Nisnevich sheaves of abelian groups of the form
\[
0 \longrightarrow {\mathbf F}_5 \longrightarrow \bpi_3^{\aone}({\mathbb A}^3 \setminus 0) \longrightarrow \mathbf{GW}^3_4 \longrightarrow 0,
\]
where $\mathbf{GW}^3_4$ is a sheafification of a certain Karoubi $U$-theory group for the Nisnevich topology, and $\mathbf{F}_5$ is a quotient of the sheaf ${\mathbf T}_5$ introduced in \cite[\S 3.6]{AsokFaselSpheres} (see Section \ref{section:metastablecomputations} for more details). Moreover, the epimorphism ${\mathbf T}_5 \to {\mathbf F}_5$ becomes an isomorphism after $4$-fold contraction.
\end{thmintro}

The word contraction in the theorem refers to a well-known construction of Voevodsky. If $\F$ is a sheaf of abelian groups, then its contraction $\F_{-1}$ is defined by $\F_{-1}(X):=\ker \F(X\times\gm{})\to \F(X)$ where the map is induced by the morphism $X\to X\times \gm {}$ defined by $x\mapsto (x,1)$.

\begin{remintro}
To be clear, this computation actually goes beyond Murthy's splitting conjecture.  If $X$ is a smooth affine $d$-fold over an infinite perfect field $k$ having characteristic unequal to $2$, and $\mathcal{E}$ is a rank $d-1$ vector bundle on $X$, then there is a primary ``Euler class" obstruction to splitting off a trivial rank $1$ summand from $\mathcal{E}$.  In the special case where $d = 4$, when this primary obstruction vanishes, our computations involving homotopy sheaves yield a secondary characteristic class whose vanishing is necessary and sufficient to establish the existence of a splitting.  In fact, the proof of Theorem \ref{thmintro:murthysplittingdimension4} amounts to showing, under the additional hypothesis that $k$ is algebraically closed, that i) the primary obstruction can be identified with the vanishing of the third Chern class, and ii) the secondary characteristic class always vanishes; we refer the reader to the beginning of Section \ref{section:splittingproblem} and Remark \ref{rem:nonalgebraicallyclosedsplitting} for a discussion of this point of view.  It would be interesting to know whether this secondary characteristic class can be described using cohomology operations. We pursue this approach in \cite{Asok13c}.
\end{remintro}

\begin{remintro}
The assumption that $k$ has characteristic unequal to $2$ in Theorem \ref{thmintro:murthysplittingdimension4} comes by way of Theorem \ref{thmintro:pi3a3minus0} and our appeal to the machinery of Grothendieck-Witt theory in its proof.  At the moment, we do not know how to remove this assumption, and thus Conjecture \ref{conj:murthy} remains open in this case.
\end{remintro}

As mentioned earlier, the space ${\mathbb A}^d \setminus 0$ is a motivic sphere, and the computation above, together with the parallels in algebraic topology, hint at an extraordinarily rich structure in its unstable $\aone$-homotopy sheaves.  The results above exemplify how this structure is reflected in the splitting problem for projective modules.  We draw the reader's attention to some tantalizing features of the above computation.  The Grothendieck-Witt sheaf $\mathbf{GW}^3_4$ that appears corresponds to the part of the $\aone$-homotopy sheaf detected by the ``degree" homomorphism in Hermitian K-theory, though we defer a detailed explanation of this connection to a subsequent paper \cite{AsokFaselKMW3}. On the other hand, the kernel of the surjective map to the Grothendieck-Witt sheaf is closely related to the motivic version of the classical $J$-homomorphism.  In a sense we will make precise (see Proposition \ref{prop:gwmodulestructureofpi35a3minus0} and Corollary \ref{cor:complexrealization}), the $24$ that appears is the ``same" $24$ that intervenes in the third stable homotopy group of the classical sphere spectrum (see \cite[Theorem 16.4]{HuHomotopy}): our computation therefore mixes together topological information about the homotopy groups of spheres and arithmetic information about the base-field and its finitely generated extensions!

The sheaf ${\mathbf I}^5$ contained in the sheaf ${\mathbf T}_5$ appearing in Theorem \ref{thmintro:pi3a3minus0} appears to be a purely unstable phenomenon (see Corollary \ref{cor:realrealization} and Remark \ref{rem:i5isalowdimensionalaccident}); detailed analysis of this phenomenon is deferred to \cite{AsokFaselSSS}.  Up to this contribution from $\mathbf{I}^5$, the sheaf $\bpi_3^{\aone}({\mathbb A}^3 \setminus 0)$ is an extension of two sheaves that are of ``stable" provenance (in the sense of stable $\aone$-homotopy theory \cite[\S 5]{MIntro}).  While we cannot yet compute the groups $\bpi_d^{\aone}({\mathbb A}^d \setminus 0)$ for $d > 3$, based on the analogy with classical unstable homotopy groups of spheres, we still expect these sheaves to stabilize in a suitable sense: they should be an extension of a (subsheaf of a) Grothendieck-Witt sheaf by an appropriate Milnor K-theory sheaf modulo $24$.  A precise conjecture on the structure of $\bpi_d^{\aone}({\mathbb A}^d \setminus 0)$ is given in \cite{AsokFaselOberwolfach}.  Moreover, the phenomenon that $\bpi_d^{\aone}({\mathbb A}^d \setminus 0)$ is an extension of two ``stable" pieces in the known examples, together with computations from classical unstable homotopy theory \cite{Toda,Mahowald}, hint at the existence of a meta-stable range for $\aone$-homotopy sheaves of ${\mathbb A}^d \setminus 0$.

\subsubsection*{Detailed description of contents}
The computation of $\bpi_3^{\aone}({\mathbb A}^3 \setminus 0)$ involves a number of ingredients, some of which are established in greater generality than actually required for the applications to projective modules envisioned in this paper.  Section \ref{section:preliminaries} begins with a short survey of some terminology and results from $\aone$-homotopy theory.  Section \ref{section:geometricbottperiodicity} quickly reviews aspects of the theory of Grothendieck-Witt groups, including geometric representability results in the $\aone$-homotopy category and closes with a cohomology computation; this computation is one of the key ingredients in the proof of Theorem \ref{thmintro:murthysplittingdimension4}.

Section \ref{section:metastablecomputations} begins by establishing a ``stable range" for the $\aone$-homotopy sheaves of $GL_{2n}/Sp_{2n}$ in the sense of \cite{BottPeriodicity}.  The main goal of the section is then to compute the first non-stable $\aone$-homotopy sheaf of $GL_{2n}/Sp_{2n}$ and to provide a brief comparison with corresponding results in classical algebraic topology.  Our use of the terminology ``non-stable" for the homotopy sheaves under consideration follows \cite{Kervaire} and is meant to indicate that we are outside the ``stable range" mentioned previously.  The computation of the first non-stable homotopy sheaf of $GL_{4}/Sp_{4}$ (really $SL_4/Sp_4$) effectively yields Theorem \ref{thmintro:pi3a3minus0}, but the result in the text (i.e., Theorem \ref{thm:pi3a3minus0}) contains a discussion of some finer points of equivariant structures that intercede in subsequent cohomological computations.

Section \ref{section:hopfmap} is devoted to analyzing the computation of $\bpi_3^{\aone}({\mathbb A}^3 \setminus 0)$ in greater detail. The fiber sequences of Section \ref{section:metastablecomputations} identify $\bpi_3^{\aone}({\mathbb A}^3 \setminus 0)$ as an extension of a Grothendieck-Witt sheaf by the sheaf we called ${\mathbf F}_5$ above.  The main goal of this section is to understand the origins of the ${\mathbf F}_5$ factor.  In a sense we make more precise, the factor of ${\mathbf F}_5$ is ``generated" by a map we call $\delta$.

Finally, Section \ref{section:splittingproblem} studies the problem of splitting a trivial rank $1$ summand off a vector bundle by means of the techniques of obstruction theory.  We first identify the primary obstruction to splitting in Proposition \ref{prop:eulerclassalgclosedisachernclass}.  For rank $d-1$ vector bundles on a smooth affine $d$-fold, we formulate a general cohomological vanishing conjecture that implies Murthy's splitting conjecture in Corollary \ref{cor:murthysconjectureisvanishing}.  Theorem \ref{thmintro:murthysplittingdimension4} is then proven by establishing the vanishing theorem alluded to above in the case $d = 4$.

\subsubsection*{Acknowledgements}
The authors would like to thank Marco Schlichting and Girja Tripathi for making available preliminary versions of their work on geometric representability of Hermitian K-theory, and V. Srinivas for helpful questions.  The first author would also like to thank Brent Doran for many discussions about homotopy theory of quadrics, which led to a better understanding of the sheaf $\mathbf{F}_5$, and Fabien Morel for an explanation several years ago as to why Milnor K-theory modulo $24$ should appear in the first stable $\aone$-homotopy sheaf of the motivic sphere spectrum.  Special thanks are due to the referees of this paper for a very careful reading and numerous comments that helped improve the exposition.

\section{Preliminaries and notation}\label{section:preliminaries}
In this section, we discuss some preliminary results that will be used in the sequel.  Fix a field $k$, which, throughout the paper, we will assume to be perfect; the latter hypothesis is necessitated by our repeated appeals to Theorem \ref{thm:morelaoneinvariance}.  The careful reader might also want to assume that $k$ is infinite because the current proofs of Theorem \ref{thm:morelaoneinvariance} rely on \cite[Lemma 1.15]{MField}, which states without proof a generalization of Gabber's presentation lemma, a result that normally assumes $k$ infinite.  In some theorem statements, e.g., those appealing implicitly or explicitly to Grothendieck-Witt theory, further assumptions on $k$, e.g., that it have characteristic unequal to $2$ will be made.  Write $\Sm_k$ for the category of schemes that are smooth, separated and have finite type over $\Spec k$.

\subsection{Preliminaries on $\aone$-homotopy theory}
\label{ss:preliminaries}
Set $\Spc_k := \Delta^{\circ}\Shv_{\Nis}(\Sm_k)$ (resp. $\Spc_{k,\bullet}$) for the category of (pointed) simplicial sheaves on the site of smooth schemes equipped with the Nisnevich topology; objects of this category will be referred to as {\em (pointed) $k$-spaces}, or simply as {\em (pointed) spaces} if $k$ is clear from the context.  Write $\hsnis$ (resp $\hspnis$) for the (pointed) Nisnevich simplicial homotopy category: this category can be obtained as the homotopy category of, e.g., the injective local model structure on $\Spc_k$ (see, e.g., \cite{MV} for details).  Write $\ho{k}$ (resp. $\hop{k}$) for the associated $\aone$-homotopy category, which is constructed as a Bousfield localization of $\hsnis$ (resp. $\hspnis$).

Equivalently, the category $\ho{k}$ can be obtained as the homotopy category of the ``motivic" model structure on the category of simplicial presheaves on $\Sm_k$ described in \cite[\S 2]{DundasOestvaerRoendigs}; this equivalent presentation will be useful at some points.  In particular, in the motivic model structure by \cite[Corollary 2.16]{DundasOestvaerRoendigs}, motivic weak equivalences and motivic fibrations with motivically fibrant codomains are closed under filtered colimits.  Thus, in this setting, formation of homotopy fibers commutes with formation of filtered colimits.

Given two (pointed) $k$-spaces $(\mathscr{X},x)$ and $(\mathscr{Y},y)$, we set $[\mathscr{X},\mathscr{Y}]_{s} := \hom_{\hsnis}(\mathscr{X},\mathscr{Y})$, $[(\mathscr{X},x),(\mathscr{Y},y)]_s := \hom_{\hspnis}((\mathscr{X},x),(\mathscr{Y},y))$, $[\mathscr{X},\mathscr{Y}]_{\aone} := \hom_{\ho{k}}(\mathscr{X},\mathscr{Y})$, and $[(\mathscr{X},x),(\mathscr{Y},y)]_{\aone} := \hom_{\hop{k}}((\mathscr{X},x),(\mathscr{Y},y))$.  Sometimes, for notational compactness it will be helpful to suppress base-points from the notation.  We write $S^i_s$ for the constant sheaf on $\Sm_k$ associated with the simplicial $i$-sphere, and $\gm{}$ will always be pointed by $1$.

If $\mathcal{X}$ is any space, the sheaf of $\aone$-connected components $\bpi_0^{\aone}({\mathcal X})$ is the Nisnevich sheaf associated with the presheaf $U \mapsto [U,\mathcal{X}]_{\aone}$.  More generally, the $\aone$-homotopy sheaves of a pointed space $(\mathscr{X},x)$, denoted $\bpi_i^{\aone}(\mathscr{X},x)$ are defined as the Nisnevich sheaves associated with the presheaves $U \mapsto [S^i_s \wedge U_+,(\mathscr{X},x)]_{\aone}$.  We also write $\bpi_{i,j}^{\aone}({\mathscr X},x)$ for the Nisnevich sheafification of the presheaf $U \mapsto [S^i_s \wedge \gm{\wedge j} \wedge U_+,({\mathscr X},x)]_{\aone}$.  We caution the reader that this convention for bigraded homotopy sheaves differs from some others in the literature.  For $i = j = 0$, the morphism of sheaves $\bpi_{0,0}^{\aone}(\mathcal{X},x) \to \bpi_0^{\aone}({\mathcal X})$ obtained by ``forgetting the base-point" is an isomorphism.

\subsection{Strictly $\aone$-invariant sheaves}
\label{ss:strictlyaoneinvariant}
Recall from \cite[Definition 1.7]{MField} that a presheaf of sets $\F$ on $\Sm_k$ is called {\em $\aone$-invariant} if for any smooth $k$-scheme $U$ the morphism $\F(U) \to \F(U \times \aone)$ induced by pullback along the projection $U \times \aone \to U$ is a bijection.  A Nisnevich sheaf of groups $\mathcal{G}$ is called {\em strongly $\aone$-invariant} if the cohomology presheaves $H^i_{\Nis}(\cdot,\mathcal{G})$ are $\aone$-invariant for $i = 0,1$.  A Nisnevich sheaf of abelian groups ${\mathbf A}$ is called {\em strictly $\aone$-invariant} if the cohomology presheaves $H^i_{\Nis}(\cdot,{\mathbf A})$ are $\aone$-invariant for every $i \geq 0$. The full-subcategory of the category of Nisnevich sheaves of abelian groups on $\Sm_k$ whose objects are strictly $\aone$-invariant sheaves is abelian by \cite[Lemma 6.2.13]{MStable}.  Typically, strictly $\aone$-invariant sheaves will be indicated in boldface.

If ${\mathbf A}$ is strictly $\aone$-invariant, it is necessarily unramified in the sense of \cite[Definition 2.1]{MField}.  It follows from, e.g.,  \cite[Theorem 2.11]{MField}, that if $f: \mathbf{A} \to \mathbf{A}'$ is a morphism of strictly $\aone$-invariant sheaves, then $f$ is an isomorphism (or epimorphism or monomorphism) if and only if the induced morphism on sections over finitely generated field extensions of the base field is an isomorphism (or epimorphism or monomorphism).

\begin{thm}[{\cite[Theorems 1.9 and 1.16]{MField}}]
\label{thm:morelaoneinvariance}
Suppose $k$ is a perfect field, and $(\mathcal{X},x)$ is any pointed $k$-space.  The sheaves $\bpi_i^{\aone}(\mathcal{X},x)$ are strongly $\aone$-invariant for $i \geq 1$.  A strongly $\aone$-invariant sheaf of abelian groups is strictly $\aone$-invariant; in particular, $\bpi_i^{\aone}(\mathcal{X},x)$ is strictly $\aone$-invariant for $i \geq 2$.
\end{thm}

\begin{ex}[Quillen K-theory]
\label{ex:quillenktheory}
For any $i\geq 0$, let $\K^Q_i$ be the Nisnevich sheaf associated with the presheaf $U \mapsto K_i(U)$, where $K_i(U)$ is the Quillen K-theory group of the smooth $k$-scheme $U$. The sheaves $\K^Q_i$ are strictly $\aone$-invariant: this can be seen in a number of ways, but, e.g., it follows from Theorem \ref{thm:morelaoneinvariance} and the representability of Quillen K-theory in $\ho{k}$, i.e., \cite[\S 4 Theorem 3.13]{MV}.
\end{ex}

\begin{ex}[Unramified sheaves]
\label{ex:unramfiedmilnorktheory}
If $U$ is a smooth scheme, with function field $k(U)$, we write $\K^M_i(U)$ for the subgroup of $K^M_i(k(U))$ consisting of elements that are unramified at each codimension $1$ point of $U$ (see, e.g., \cite[\S 2.2]{MMilnor} for more details).  The assignment $U \mapsto \mathbf{K}^M_i(U)$ defines a sheaf $\K^M_i$ for the Nisnevich topology.  That the sheaf $\K^M_i$ is strictly $\aone$-invariant is proven in \cite[Corollary 6.5, Proposition 8.6]{Rost96}.  It follows that $\K^M_i/n$, which is the cokernel of the multiplication by $n$ map, is again a strictly $\aone$-invariant sheaf.
\end{ex}

\begin{ex}[Unramified sheaves related to Witt groups]
\label{ex:unramifiedwitt}
In the same spirit, the presheaf $U \mapsto {W}_{ur}(U)$ of unramified Witt groups is a sheaf for the Nisnevich topology, which is strictly $\aone$-invariant by \cite[Theorem 1.1]{Panin10}; we write $\mathbf{W}$ for this sheaf. For any integer $n$, the assignments $U \mapsto I^n_{ur}(U) \subset {W}_{ur}(U)$ also define Nisnevich sheaves $\mathbf{I}^n$ (when $n = 1$, we drop the superscript); we refer to these sheaves as unramified powers of the fundamental ideal, and they are strictly $\aone$-invariant by \cite[Th\'eor\`eme 11.2.9]{FaselChowWitt}.
\end{ex}

\begin{ex}[Unramified Milnor-Witt K-theory sheaves]
\label{ex:unramifiedmilnorwitt}
Morel defines unramified Milnor-Witt K-theory sheaves $\K^{MW}_n$ in \cite[\S 3.2]{MField}.  It follows from \cite[Theorems 3.37 and 3.46]{MField} that these sheaves are strictly $\aone$-invariant in case $n\geq 0$.  When $n < 0$, the sheaf $\K^{MW}_n$ is isomorphic to $\mathbf{W}$ and strict $\aone$-invariance follows from the result of Panin quoted in the previous paragraph.
\end{ex}

\subsection{Contracted sheaves}
\label{ss:contracted}
For any smooth $k$-scheme $U$, the unit map $\Spec k \to \gm{}$ yields a morphism $U\to U\times\gm{}$.  If $\mathbf{A}$ is a strictly $\aone$-invariant sheaf, the sheaf $\mathbf{A}_{-1}$ is then defined by
\[
\mathbf{A}_{-1}(U) := \ker(\mathbf{A}(\gm{} \times U) \to \mathbf{A}(U)).
\]
We can then inductively define $\mathbf{A}_{-n} :=(\mathbf{A}_{-n+1})_{-1}$ for any integer $n\geq 1$; we call $\mathbf{A}_{-n}$ the $n$-th contraction of $\mathbf{A}$.  Contraction is an exact functor on the category of strictly $\aone$-invariant sheaves (in fact, slightly more is true; see, e.g., \cite[Lemma 7.33]{MField}). If $({\mathcal X},x)$ is a pointed $\aone$-connected space, the $\aone$-homotopy sheaves of $\gm{}$-loop spaces of ${\mathcal X}$ are related to those of ${\mathcal X}$ by the following result of Morel.

\begin{thm}[{\cite[Theorem 6.13]{MField}}]
If $(\mathscr{X},x)$ is a pointed $\aone$-connected space, then for every pair of integers $i,j \geq 1$
\[
\bpi^{\aone}_{i,j}( \mathscr{X},x):= \bpi_i^{\aone}({\mathbf R}\Omega^j_{\gm{}}(\mathscr{X},x)) = \bpi_i^{\aone}(\mathscr{X},x)_{-j}.
\]
\end{thm}

Given any smooth scheme $X$ and any strictly $\aone$-invariant sheaf $\mathbf{A}$, Morel defines a complex called the Rost-Schmid complex \cite[Definition 5.7 and Theorem 5.31]{MField} whose terms are expressed using successive contractions of $\mathbf{A}$.  Morel shows that the Rost-Schmid complex coincides with the Gersten complex \cite[Corollary 5.44]{MField} and we will routinely use this identification in the sequel.  In particular, all the strictly $\aone$-invariant sheaves in the previous section admit Gersten resolutions (presented in a particularly nice form) by this result.

\begin{ex}
\label{ex:classicalcontractions}
The descriptions of the contractions of the sheaves mentioned in Examples \ref{ex:quillenktheory}-\ref{ex:unramifiedmilnorwitt} are classical.  In particular, $(\K^Q_i)_{-j} \cong \K^Q_{i-j}$, $(\K^M_i)_{-j} \cong \K^M_{i-j}$, $(\mathbf{I}^i)_{-j} \cong \mathbf{I}^{i-j}$, $(\mathbf{W})_{-j} \cong \mathbf{W}$ and $(\K^{MW}_i)_{-j} \cong \K^{MW}_{i-j}$.  Proofs of these results can be found, e.g., in \cite[Lemma 2.7 and Proposition 2.9]{AsokFaselSpheres}.
\end{ex}

\subsection{Twists of strictly $\aone$-invariant sheaves}
\label{ss:twists}
In formulating the obstructions to lifting, we will, in analogy with the situation in classical topology, be forced to consider the analog of cohomology with coefficients in a local system.  In the context of the splitting problem, the ``local systems" that appear are ``line bundle twists of strictly $\aone$-invariant sheaves."

If $F/k$ is a field extension, and $a \in F^{\times}$, following the discussion of \cite[p. 51]{MField} we set $\langle a \rangle = 1 + \eta [a]$.  The assignment $a \mapsto \langle a \rangle$ defines, by \cite[Lemma 3.5(4)]{MField}, for any smooth connected $k$-scheme $U$ a homomorphism $\gm{}(k(U)) \to \K^{MW}_0(k(U))^{\times}$ and a morphism of sheaves $\gm{} \to (\K^{MW}_0)^{\times}$.  This morphism extends to a morphism of sheaves of rings $\Z[\gm{}] \to \K^{MW}_0$.

Suppose $\mathbf{A}$ is a strictly $\aone$-invariant sheaf.  Any sheaf of the form $\mathbf{A}_{-1}$ is equipped with an action of $\gm{}$ which factors through an action of $\K_0^{MW}$ via the morphism $\gm{} \to {\K^{MW}_0}^{\times}$ mentioned above \cite[Lemma 3.49]{MField}. For instance, since $(\K_n^{MW})_{-1}=\K_{n-1}^{MW}$ for any $n\in\Z$ \cite[\S 3.2]{MField}, the action $\K_0^{MW}\times \K_{n-1}^{MW}\to \K_{n-1}^{MW}$ is the map induced by multiplication in Milnor-Witt K-theory \cite[p. 80]{MField}.

Assume that $\mathbf{A}$ carries an action of $\gm{}$. If $X$ is a smooth scheme and $\L$ is an invertible $\O_X$-module with associated line bundle $L$, write $L^{\circ}$ for the $\gm{}$-torsor underlying $L$, i.e., the complement of the zero section in $L$. Set
\[
\mathbf{A}(\L) := \Z[L^{\circ}]\otimes_{\Z[\gm{}]}\mathbf{A}.
\]
Sheaves of the form $\mathbf{A}(\L)$ admit a Gersten (or Rost-Schmid) resolution following \cite[Remarks 5.13 and 5.14]{MField}. If $\mathbf{A}=\K_n^{MW}$ equipped with the $\gm{}$-action induced by multiplication by $\K_0^{MW}$, and $\L$ is a line bundle on a smooth scheme $X$, then the Gersten resolution of $\K_n^{MW}(\L)$ is the one considered in \cite[\S 5]{AsokFaselThreefolds} and coincides with the fiber product of \cite[D\'efinition 10.2.7]{FaselChowWitt}.

\section{Grothendieck-Witt sheaves and geometric Bott periodicity}\label{section:geometricbottperiodicity}
In this section, we begin by briefly recalling some facts about Grothendieck-Witt groups, a.k.a. Hermitian $K$-theory, which is the algebro-geometric analog of topological $KO$-theory. Our main sources are \cite{Schlichting09,Schlichting10,SchlichtingTripathi}.  From the standpoint of this paper, the most important fact about these groups that we will use is that Grothendieck-Witt groups are representable in the $\aone$-homotopy category by explicit ind-algebraic varieties, which provide algebro-geometric models for the spaces appearing in the classical Bott periodicity theorem.  The $\aone$-homotopy sheaves of the spaces representing Grothendieck-Witt groups are denoted $\mathbf{GW}^j_i$ and we describe and study several equivalent versions of a natural action of $\gm{}$ on these sheaves.  Explicit information about this action is required in the proof of Theorem \ref{thmintro:murthysplittingdimension4} because we will need to consider cohomology with coefficients in Grothendieck-Witt sheaves twisted by a line bundle as described in Subsection \ref{ss:twists}.  We end our discussion with a computation of $H^n_{\Nis}(X,\mathbf{GW}_n^{n-1})$ (and its twisted versions) for any smooth scheme $X$.

\subsection{Definitions}
\label{ss:definitions}
Let $X$ be a smooth scheme over a field $k$ of characteristic different from $2$ (we keep these assumptions throughout the section, though it is not necessary for some of the arguments).  For every such $X$, and any line bundle $\L$ on $X$, one has an {\em exact category with weak equivalences and (strong) duality} $(\mathcal{C}^b(X),qis,\sharp_{\L},\varpi_{\L})$ in the sense of \cite[\S 2.3]{Schlichting10}.  Here, $\mathcal{C}^b(X)$ is the category of bounded complexes of locally free coherent sheaves on $X$, weak equivalences are given by quasi-isomorphisms of complexes, $\sharp_{\L}$ is the duality functor induced by the functor $\hom_{\O_X}(\_,\L)$, and the natural transformation $\varpi_{\L}:1\to \sharp_{\L}\sharp_{\L}$ is induced by the canonical identification of a locally free sheaf with its double dual.  The (left) translation (or shift) functor $T:\mathcal{C}^b(R)\to \mathcal{C}^b(R)$ yields new dualities $\sharp^n_{\L}:=T^n\circ \sharp_{\L}$ and canonical isomorphisms $\varpi^n_{\L}:=(-1)^{n(n+1)/2}\varpi_{\L}$.

Associated with an exact category with weak equivalences and (strong) duality is a Grothendieck-Witt space and higher Grothendieck-Witt groups, obtained as homotopy groups of the Grothendieck-Witt space \cite[\S 2.11]{Schlichting10}.  We write $\mathcal{GW}(\mathcal{C}^b(X),qis,\sharp^j_{\L},\varpi^j_{\L})$ for the Grothendieck-Witt space of the example described in the previous paragraph.

\begin{defn}
\label{defn:gwgroups}
For $i\geq 0$, we denote by $GW^j_i(X,{\L})$ the group $\pi_i\mathcal{GW}(\mathcal{C}^b(X),qis,\sharp^j_{\L},\varpi^j_{\L})$. If ${\L}=\O_X$, we write $GW^j_i(X)$ for $GW^j_i(X,\O_X)$.
\end{defn}

Since $2$ is invertible in $k$ by assumption, the Grothendieck-Witt groups defined above are a Waldhausen-style version of hermitian $K$-theory as defined by M. Karoubi \cite{Karoubi,Karoubi80} in the case of affine schemes (see \cite[Remark 4.16]{Schlichting09} or \cite{Horn}).  In particular, given a smooth $k$-algebra $R$ there are identifications of the form
\[
\begin{split}
GW_i^0(R) &= K_iO(R), \text{ and } \\
GW_i^2(R) &= K_iSp(R)
\end{split}
\]
by \cite[Corollary A.2]{SchlichtingHKDE}.  Furthermore, there are identifications $GW_i^1(R)={}_{-1}U_i(R)$ and $GW_i^3(R)=U_i(R)$, where the groups $U_i(R)$ and ${}_{-1}U_i(R)$ are Karoubi's $U$-groups, and $GW_i^n$ is $4$-periodic in $n$.

We will frequently use the \emph{Karoubi periodicity} exact sequences \cite[Theorem 6.1]{SchlichtingHKDE}:
\[
\xymatrix@R=.5em@C=2.5em{\ldots\ar[r]^-{f_{i,j-1}} & K_i(X)\ar[r]^-{H_{i,j}} & GW_i^j(X,{\L})\ar[r]^-{\eta} & GW_{i-1}^{j-1}(X,{\L})\ar[r]^-{f_{i-1,j-1}} & K_{i-1}(X)\ar[r]^-{H_{i-1,j}} & \ldots
}
\]
which are defined for any $i,j\in \mathbb{N}$ and any line bundle ${\L}$ over $X$. Here, the maps $H_{i,j}:K_i(X)\to GW_i^j(X,{\L})$ are the hyperbolic morphisms, $f_{i,j}:GW_i^j(X,{\L})\to K_i(X)$ are the forgetful morphisms and $\eta$ is multiplication by a distinguished element in $GW_{-1}^{-1}(k)$.

\begin{rem}
If $F$ is a field, Morel shows in \cite[Lemma 3.10]{MField} that there are canonical isomorphisms $\K^{MW}_0(F) \isomt GW(F)$ and $\K^{MW}_{-i}(F) \isomt W(F)$ for $i > 0$.  By definition, $GW^0_0(F) = GW(F)$ and $GW^{-i}_{-i}(F) \cong W(F)$.  Via these identifications $\eta \in K^{MW}_{-1}(F)$ coincides with $\eta \in GW^{-1}_{-1}(F)$.
\end{rem}

\subsection{Classifying spaces}
\label{ss:classifyingspaces}
If $G$ is a simplicial Nisnevich sheaf of groups, we write $BG$ for the usual simplicial classifying space of $G$ (see \cite[\S 4.1]{MV}) and $\mathcal BG$ for a fibrant model of $BG$.  As usual, $BG$ can be described as $EG/G$ where $EG$ is a simplicially contractible space with a free action of $G$.

If $G$ is a linear algebraic group, then by \cite[\S 4 Proposition 1.15]{MV}, the space $BG$ classifies Nisnevich locally trivial $G$-torsors in $\hsnis$.  In particular, if $P \to X$ is a Nisnevich locally trivial $G$-torsor over a smooth scheme $X$, there is a (well-defined up to simplicial homotopy) morphism $X \to \mathcal {B}G$ such that $P$ is the pullback of the universal $G$-torsor over $\mathcal {B}G$ \cite[\S 4, Proposition 1.15]{MV}.

\subsection{Geometric representability}
\label{ss:geometricrepresentability}
For any $j\in\N$, write $\mathscr{GW}^j$ for a simplicially fibrant model of the presheaf $X\mapsto \mathcal{GW}(\mathcal{C}^b(X),qis,\sharp^j,\varpi^j)$. As Grothendieck-Witt groups satisfy Nisnevich descent and are $\aone$-homotopy invariant, the spaces $\mathscr{GW}^j$ represent Grothendieck-Witt groups in $\hop k$, i.e.,
\[
GW_i^j(X)=[S^i_s\wedge X_+,\mathscr{GW}^j]_{\aone}
\]
for any $i,j\in\N$ (\cite[Theorem 3.1]{Hornbostel05} or \cite[Lemma 6.2]{PaninWalter1}). For any $i,j\in\N$, write $\mathbf{GW}_i^j$ for the Nisnevich sheaf on $\Sm_k$ associated with the presheaf $U \mapsto GW_i^j(U)$. The above formula immediately yields $\piaone_i(\mathscr{GW}^j)=\mathbf{GW}_i^j$. We have therefore obtained the following result.  We now review the explicit geometric model for the space $\Omega_s^1\mathscr{GW}^3$ provided by \cite{SchlichtingTripathi}.

\begin{rem}[Quotients]
In the sequel, we will discuss quotients of smooth schemes $X$ by the free action of a (reductive) linear algebraic group $G$ . A priori, there are two things we could mean by ``quotient."  First, one could consider the quotient of the representable sheaf $X$ by the action of the sheaf of groups $G$ as a Nisnevich sheaf.  Alternatively, if $X$ is a smooth quasi-projective variety and $G$ is a reductive algebraic group acting (scheme-theoretically) freely on $X$, then a geometric quotient $X/G$ exists as a smooth scheme \cite[Chapter 1]{GIT}.  If $G$ is special, i.e., all $G$-torsors are Zariski locally trivial, which happens for $G = GL_n, SL_n$ or $Sp_{2n}$, then the quotient map $X \to X/G$ is Zariski locally trivial.  Using Zariski local sections one deduces that the quotient of the representable sheaf $X$ by the representable sheaf $G$ computed in the category of Nisnevich sheaves coincides with the representable sheaf associated with $X/G$, i.e., the notation $X/G$ is unambiguous.
\end{rem}

As usual, write $GL:=\colim_{m\in \N} GL_m$ where the transition maps $GL_m\hookrightarrow GL_{m+1}$ are obtained by sending an invertible $m \times m$ matrix $M$ to the block $(m+1)\times (m+1)$-matrix
\[
\begin{pmatrix}M & 0 \\ 0 & 1 \end{pmatrix}.
\]
By \cite[\S 4 Theorem 3.13]{MV}, the space $KGL:=\Z \times BGL$ represents algebraic K-theory in $\hop{k}$.

Similarly, write $Sp:=\colim_{m\in\N} Sp_{2m}$ where the maps $Sp_{2m}\to Sp_{2m+2}$ are obtained my mapping a symplectic $2m\times 2m$ matrix $M$ to the block $(2m+2)\times (2m+2)$-symplectic matrix
\[
\begin{pmatrix}M & 0 \\ 0 & Id_2 \end{pmatrix}.
\]
By \cite[Theorem 8.2]{PaninWalter1}, the space $KSp:=\Z\times BSp$ represents symplectic $K$-theory in $\hop{k}$. Taking the colimit of the standard inclusions $Sp_{2m}\to GL_{2m}$ yields a map $Sp\to GL$ and therefore a map $BSp\to BGL$. Taking the cartesian product with the multiplication by $2$ map $\Z\to \Z$, we obtain a map $KSp\to KGL$ which we refer to as the \emph{forgetful} map.

It follows from \cite[Proposition 5.2]{Wendt} that the inclusions $Sp_{2m}\to GL_{2m}$ yield $\aone$-fiber sequences
\[
GL_{2m}/Sp_{2m} \longrightarrow BSp_{2m} \longrightarrow BGL_{2m}
\]
for any $m\in\N$ and thus, passing to the colimits and using \cite[Corollary 2.16]{DundasOestvaerRoendigs}, an $\aone$-fiber sequence
\[
GL/Sp \longrightarrow BSp \longrightarrow BGL.
\]
By \cite[Theorem 8.4]{SchlichtingTripathi}, the space $GL/Sp$ is isomorphic to $\Omega_s^1\mathscr{GW}^3$ in $\hop{k}$.

\begin{prop}
\label{prop:geometricGW3}
For any $i\in\N$, we have $\piaone_i(GL/Sp)=\mathbf{GW}_{i+1}^3$.
\end{prop}

It follows from this Proposition and Theorem \ref{thm:morelaoneinvariance} that the sheaf $\mathbf{GW}^j_i$ is strictly $\aone$-invariant for $i \geq 2$.  In fact, the sheaves $\mathbf{GW}^j_i$ are strictly $\aone$-invariant for $i < 2$ as well; this follows, e.g., from the stable representability of Grothendieck-Witt groups \cite{Hornbostel05,PaninWalter1}.

\subsection{Contractions of Grothendieck-Witt sheaves}
\label{ss:contractions}
Recall the definition of the contraction of a sheaf from Subsection \ref{ss:contracted}.

\begin{lem}[{\cite[Proposition 4.4]{AsokFaselThreefolds}}]
\label{lem:contractionofgw}
For any $i,j\in\mathbb{Z}$, we have $(\mathbf{GW}^j_{i})_{-1}=\mathbf{GW}^{j-1}_{i-1}$.
\end{lem}

\begin{lem}\label{lem:karoubiperiodicitycontractions}
If $H_{i,j}:\K^Q_i\to \mathbf{GW}_i^j$ is the hyperbolic morphism, and $f_{i,j}:\mathbf{GW}_i^j\to \K_i^Q$ is the forgetful morphism, then $(H_{i,j})_{-1}=H_{i-1,j-1}$ and $(f_{i,j})_{-1}=f_{i-1,j-1}$.
\end{lem}

\begin{proof}
The proof of \cite[Theorem 9.13]{SchlichtingHKDE} relies on the Mayer-Vietoris long exact sequence \cite[Theorem 9.6]{SchlichtingHKDE}. It suffices to observe that both the hyperbolic and the forgetful homomorphisms induce morphisms of Mayer-Vietoris sequences, which is a consequence of the construction.
\end{proof}

While the negative K-theory of a smooth scheme is trivial, smooth schemes can have non-trivial negative Grothendieck-Witt groups.  As a consequence, iterated contractions of Grothendieck-Witt sheaves need not be trivial.  The next result shows that, nevertheless, for certain values of the indices, repeated contraction of Grothendieck-Witt sheaves does produce the trivial sheaf.

\begin{prop}
\label{prop:vanishingofcontractions}
If $j > i+1$, then $(\mathbf{GW}^i_{i+1})_{-j} = 0$
\end{prop}

\begin{proof}
By Lemma \ref{lem:contractionofgw}, we know that $(\mathbf{GW}^i_{i+1})_{-i-1} \cong \mathbf{GW}^3_0$.  To establish the proposition, it suffices to show that the contractions of $\mathbf{GW}^3_0$ are trivial. To this end, recall that the hyperbolic map $\K^Q_0 \to \mathbf{GW}^3_0$ becomes surjective on sections over fields by \cite[Lemma 4.1]{Fasel08c}.  Since the morphism $\K^Q_0 \to \mathbf{GW}^3_0$ is a morphism of strictly $\aone$-invariant sheaves it is thus an epimorphism (see Subsection \ref{ss:strictlyaoneinvariant}).  The claim now follows from exactness of contractions and the fact that contractions of a constant sheaf (here, $\K^Q_0=\Z$) are trivial.
\end{proof}

\subsection{Actions of $\gm{}$}
\label{ss:actionsongwsheaves}
For any scheme $X$ and any integers $i,j\in\N$, there is a multiplication map
\[
GW_0^0(X)\times GW_i^j(X) \longrightarrow GW_i^j(X)
\]
defined in \cite[\S 5.4]{SchlichtingHKDE}.  The map associating to a unit $a \in \O_X(X)^{\times}$ the bilinear form $\O_X\times \O_X\to \O_X$ defined by $(x,y)\mapsto axy$ yields a morphism $\O_X(X)^{\times} \to GW^0_0(X)^{\times}$.  Composing these two maps defines, functorially, an action of the sheaf $\gm{}$ on the sheaves $\mathbf{GW}_i^j$; we refer to this action as the {\em multiplicative action} of $\gm{}$ on $\mathbf{GW}^j_i$.

The $\gm{}$-action on ${\mathbf {GW}}_i^j$ deduced from the natural action on the contracted sheaf $({\mathbf {GW}}_{i+1}^{j+1})_{-1}$ via the isomorphism of Lemma \ref{lem:contractionofgw} coincides with the action of the previous paragraph \cite[Lemma 4.6]{AsokFaselThreefolds}.  It follows from this that the twisted Gersten complex of $\mathbf{GW}^j_i(\L)$ as discussed in \cite[Remark 4.13-14]{MField}, coincides with the one defined in \cite[Theorem 25, Proposition 28]{FaselSrinivas}.  One can also define a sheaf on the small \'etale site of $X$ as the Nisnevich sheaf associated with the presheaf $U\mapsto GW_i^j(U,\L\vert_{U})$.  It follows from the references just provided, that this sheaf also coincides with $\mathbf{GW}_i^j(\L)$ and thus the notation is unambiguous.

We now provide an alternate ``more geometric" action of $\gm{}$ on $\mathbf{GW}_i^3$. For any ring $R$ and any unit $\lambda\in R^\times$, write $\gamma_{2n,\lambda}$ for the invertible $2n\times 2n$ matrix defined inductively by setting $\gamma_{2,\lambda}:=\mathrm{diag}(\lambda,1)$ and $\gamma_{2n,\lambda}:=\gamma_{2n-2,\lambda}\perp \gamma_{2,\lambda}$. Conjugation by $\gamma_{2n,\lambda}^{-1}$ yields an action of $\gm{}$ on $GL_{2n}$ for any $n\in\N$.

The transition maps $GL_{2n}\to GL_{2n+2}$ defined above are equivariant for this action, it follows that $GL=\colim_n GL_{2n}$ is also endowed with an action of $\gm{}$. Since conjugation by $\gamma_{2n,\lambda}^{-1}$ preserves $Sp_{2n}(R)$, there is an induced action of $\gm{}$ on the quotient $GL_{2n}/Sp_{2n}$ for which the quotient map is $\gm{}$-equivariant. Again, the transition maps $GL_{2n}/Sp_{2n} \to GL_{2n+2}/Sp_{2n+2}$ are $\gm{}$-equivariant and it follows that both $Sp$ and $GL/Sp$ are endowed with $\gm{}$-actions.  In particular, since the homotopy sheaves of $GL/Sp$ are the Grothendieck-Witt sheaves $\mathbf{GW}^3_i$ by Proposition \ref{prop:geometricGW3}, passing to homotopy sheaves we obtain action maps
\[
\gm{}\times \mathbf{GW}_i^3 \longrightarrow \mathbf{GW}_i^3
\]
for any $i\geq 1$.  We will refer to this action of $\gm{}$ on $\mathbf{GW}^3_i$ as the {\em conjugation action}.  The next result shows that the two actions considered above coincide.

\begin{prop}
\label{prop:multiplicativeGLSP}
The multiplicative action and conjugation action of $\gm{}$ on $\mathbf{GW}^3_i$ coincide.
\end{prop}

\begin{proof}
In rough outline the proof goes as follows: we identify $GL/Sp$ as the simplicial homotopy fiber of a map $BSp \to BGL$, describe the multiplicative and conjugation actions on $BSp$ and $BGL$ individually, and construct a homotopy between these two actions compatible with the map $BSp \to BGL$.  This homotopy may be lifted to the simplicial homotopy fiber to provide a homotopy between the multiplication and conjugation actions there.

First, we recall the construction of another model of the space $BGL$.  To this end, suppose $R$ is a smooth $k$-algebra.  Write $\mathscr{F}_{ev}$ for the symmetric monoidal category whose objects are even-rank (finitely generated) free $R$-modules and where morphisms are isomorphisms.  The inclusion $\iota: \mathscr{F}_{ev} \hookrightarrow {\mathscr P}$ into the symmetric monoidal category $\mathscr{P}$ of all (finitely generated) projective $R$-modules (again, morphisms are isomorphisms) is cofinal \cite[p. 337]{Weibel} since every projective module is a direct summand of an even rank free module.  We can form the group completions $B\mathscr{F}_{ev}^{-1}\mathscr{F}_{ev}$ and $B \mathscr{P}^{-1}\mathscr{P}$ \cite[Chapter IV Definitions 4.2-3]{Weibel} and the map induced by the inclusion functor $\iota$ is an isomorphism on higher homotopy groups by \cite[Cofinality Theorem 4.11]{Weibel}.

Write $(\mathscr{F}_{ev}^{-1}\mathscr{F}_{ev})_0$ for the connected component of the base-point in $B\mathscr{F}_{ev}^{-1}\mathscr{F}_{ev}$.  The map $BGL \to (\mathscr{F}_{ev}^{-1}\mathscr{F}_{ev})_0$ is the map $\operatorname{hocolim}_{n \in {\mathbb N}}BAut(R^{\oplus 2n}) \to (\mathscr{F}_{ev}^{-1}\mathscr{F}_{ev})_0$ defined on objects by $n \mapsto (R^{\oplus 2n},R^{\oplus 2n})$ and on morphisms $f \in Aut(R^{\oplus 2n})$ by $f \mapsto (1,f)$.  The functor $\iota$ induces a map $B\mathscr{F}_{ev}^{-1}\mathscr{F}_{ev} \to B \mathscr{P}^{-1}\mathscr{P}$.  If $(\mathscr{P}^{-1}\mathscr{P})_0$ is the connected component of the base-point in $B \mathscr{P}^{-1}\mathscr{P}$, the induced map $(\mathscr{F}_{ev}^{-1}\mathscr{F}_{ev})_0 \to (\mathscr{P}^{-1}\mathscr{P})_0$ is a weak equivalence (since projective modules over smooth local $k$-algebras are free) \cite[Theorem 4.9 and Corollary 4.11.1]{Weibel}.  Note: the map $B\mathscr{F}_{ev}^{-1}\mathscr{F}_{ev} \to B \mathscr{P}^{-1}\mathscr{P}$ is not a weak equivalence because the induced map on $\pi_0$ corresponds to the inclusion $2 \Z \subset \Z$.

Next, we recall a construction of a model for the space $BSp$.  Write $h_{2n}$ for the standard hyperbolic space of dimension $2n$ over $R$.  Write $\mathscr{S}$ for the symmetric monoidal category whose objects are $h_{2n}$ and where morphisms are isometries.  Write $\mathscr{M}$ for the symmetric monoidal category whose objects are pairs $(P,\varphi)$, where $P$ is a (finitely generated) projective $R$-module equipped with a non-degenerate skew-symmetric form $\varphi$ and with morphisms given by isometries and, as above, $\iota: \mathscr{S} \to \mathscr{M}$ for the inclusion functor.  By \cite[Theorem A.1]{SchlichtingHKDE}, the groups $GW^2_i(R)$ can be identified as the homotopy groups of $B{\mathscr M}^{-1}\mathscr{M}$.  As above, write $({\mathscr M}^{-1}\mathscr{M})_0$ for the connected component of the base-point and repeating the previous discussion, the map $c: BSp \to (\mathscr{S}^{-1}\mathscr{S})_0$ is the map $\operatorname{hocolim}_{n \in \mathbb{N}} BAut(h_{2n}) \to (\mathscr{S}^{-1}\mathscr{S})_0$ defined on objects by $n \mapsto (h_{2n},h_{2n})$ and takes $f \in Aut(h_{2n})$ to $(1,f)$.  Again, the map $B\mathscr{S}^{-1}\mathscr{S} \to B{\mathscr M}^{-1}\mathscr{M}$ is a weak equivalence on smooth local $k$-algebras $R$ because all non-degenerate skew-symmetric forms over $R$ are hyperbolic.

Now, we argue as in \cite[Lemma 4.7]{AsokFaselThreefolds}.  First, observe that the multiplication action on $BSp$ can be given at the space level by the following construction.  For $t \in R^{\times}$, we define a strict symmetric monoidal functor $\cdot \tensor \langle t \rangle: \mathscr{M} \to \mathscr{M}$ on objects by $(P,\varphi) \tensor \langle t \rangle := (P,t\varphi)$ and act by the identity on morphisms.  The induced action of $t \in R^{\times}$ on $({\mathscr M}^{-1}\mathscr{M})_0$, which can be identified with $(\mathscr{S}^{-1}\mathscr{S})_0$ sends $(h_{2n},h_{2n})$ to $(h_{2n} \tensor \langle t \rangle,h_{2n} \tensor \langle t \rangle)$ and $(f,g) \mapsto (f,g)$.

Second, recall the homotopy between the multiplication and conjugation actions on $BSp$.  To this end, consider the diagram
\[
\xymatrix{
BSp \ar[r] \ar[d]^c& BSp \ar[d]^c \\
(\mathscr{S}^{-1}\mathscr{S})_0 \ar[r] & (\mathscr{S}^{-1}\mathscr{S})_0,
}
\]
where the top map is induced by conjugation with $\gamma_{2n,t}$ and the bottom map is the action of $t$ on $(\mathscr{S}^{-1}\mathscr{S})_0$ described above.  While this diagram is not strictly commutative, it is commutative up to natural isomorphism as observed in \cite[Lemma 4.7]{AsokFaselThreefolds}.  Thus, the action of conjugation by $\gamma_{2n,t}$ is homotopic to the multiplication by $t$.

Third, observe that the multiplication action on $BGL$ is trivial.  Indeed, there is a symmetric monoidal functor $\mathscr{S} \to \mathscr{F}_{ev}$ given by forgetting the hyperbolic form.  The composite of $\cdot \tensor \langle t \rangle$ with the forgetful functor is the identity functor on $\mathscr{F}_{ev}$.  The argument used in the previous paragraph, repeated in the context of $BGL$, shows that the conjugation action on $BGL$ is trivial up to homotopy.

Finally, observe that one can take as a model of $GL/Sp$ the simplicial homotopy fiber of the map $(\mathscr{S}^{-1}\mathscr{S})_0 \to (\mathscr{F}_{ev}^{-1}\mathscr{F}_{ev})_0$ by \cite[Lemma 8.3 and Theorem 8.4]{SchlichtingTripathi}.  We thus obtain a description of the multiplication action on $GL/Sp$ by restriction.  By the compatibility of the actions just constructed with the forgetful map, the homotopy between the multiplication and conjugation actions on $BSp$ restricts to a homotopy between the corresponding actions on $GL/Sp$, which provides the required identification.
\end{proof}

\subsection{Bounding $H^n_{\mathrm{Nis}}(X,\mathbf{GW}_n^{n-1}(\L))$}
\label{ss:boundinghngw}
Let $X$ be a smooth scheme of dimension $d$ and let $\L$ be a line bundle over $X$.  We now compute the cohomology group $H^n_{\mathrm{Nis}}(X,\mathbf{GW}_n^{n-1}(\L))$ following the lines of \cite[\S 4]{AsokFaselThreefolds}, i.e., by analyzing the Gersten resolution of $\mathbf{GW}^{n-1}_n(\L)$.

Sheafifying the Karoubi periodicity sequences for the Nisnevich topology, we obtain a long exact sequence of sheaves of the form
\begin{equation}\label{eqn:Karoubi}
\xymatrix@C=3.5em{\K_n^Q\ar[r]^-{H_{n,n-1}} & \mathbf{GW}_n^{n-1}(\L)\ar[r]^-\eta & \mathbf{GW}_{n-1}^{n-2}(\L)\ar[r]^-{f_{n-1,n-2}} & \K_{n-1}^Q.
}
\end{equation}
Observe that the sheaves that appear in the above sequence admit Gersten resolutions by the discussion of Subsection \ref{ss:twists}.

We set
\[
\begin{split}
\mathbf{A}(\L) &:= \operatorname{Im}(H_{n,n-1}), \text{ and } \\
\mathbf{B}(\L) &:= \operatorname{Im}(\eta).
\end{split}
\]
Again, the sheaves $\mathbf{A}(\L)$ and $\mathbf{B}(\L)$ admit Gersten-resolutions.

Taking $n=3$ in (\ref{eqn:Karoubi}), we obtain an exact sequence of the form:
\[
\xymatrix@C=3.5em{\K_3^Q\ar[r]^-{H_{3,2}} & \mathbf{GW}_3^{2}(\L)\ar[r]^-\eta & \mathbf{GW}_{2}^{1}(\L)\ar[r]^-{f_{2,1}} & \K_2^Q};
\]
this sequence is precisely \cite[Sequence (4.3)]{AsokFaselThreefolds} and this observation is used implicitly in the subsequent proofs.

\begin{notation}
For any smooth scheme $X$ and any integer $n\in\N$, set $Ch^n(X) := CH^n(X)/2$ where $CH^n(X)$ is the Chow group of codimension $n$ cycles on $X$.
\end{notation}

If $\K^Q_n/2$ is the mod $2$ Quillen K-theory sheaf, i.e., the cokernel of the multiplication by $2$ map $\K^Q_n \to \K^Q_n$, then $H^n(X,\K^Q_n/2) \cong Ch^n(X)$ via Bloch's formula.

\begin{lem}
\label{lem:aofl}
The hyperbolic morphism $H_{n,n-1}:\K_n^Q\to \mathbf{A}(\L)$ induces an isomorphism $H^n_{\mathrm{Nis}}(X,\K_n^Q/2)\cong  H^n_{\mathrm{Nis}}(X,\mathbf{A}(\L))$.
\end{lem}

\begin{proof}
The proof of \cite[Lemma 4.13]{AsokFaselThreefolds} applies mutatis mutandis.
\end{proof}

\begin{lem}
\label{lem:bofl}
The $(n-3)$rd contraction of the hyperbolic morphism
\[
H_{n-1,n-2}:\K_{n-1}^Q \longrightarrow \mathbf{GW}_{n-1}^{n-2}(\L)
\]
induces a morphism $H^\prime_{2,1}:\K_{2}^Q\to \mathbf{B}(\L)_{3-n}$, which in turn induces an isomorphism $Ch^{n-1}(X)\cong H^{n-1}_{\mathrm{Nis}}(X,\K_{n-1}^Q/2)\cong H^{n-1}_{\mathrm{Nis}}(X,\mathbf{B}(\L))$. Moreover, $H^{n}_{\mathrm{Nis}}(X,\mathbf{B}(\L))=0$.
\end{lem}

\begin{proof}
This follows from \cite[Lemmas 4.13-4.15]{AsokFaselThreefolds}.
\end{proof}

By definition, there is a short exact sequence of strictly $\aone$-invariant sheaves of the form:
\[
0 \longrightarrow \mathbf{A}(\L) \longrightarrow \mathbf{GW}_n^{n-1}(\L) \longrightarrow \mathbf{B}(\L)\to 0.
\]
Taking cohomology and applying Lemmas \ref{lem:aofl} and \ref{lem:bofl}, one obtains the following result.

\begin{prop}
\label{prop:surjectionfromchowmod2}
If $k$ is a field having characteristic unequal to $2$, and if $X$ is a smooth $k$-scheme, there is an exact sequence of the form
\[
Ch^{n-1}(X) \stackrel{\partial}{\longrightarrow} Ch^n(X) \longrightarrow H^n_{\mathrm{Nis}}(X,\mathbf{GW}_n^{n-1}(\L))\longrightarrow 0.
\]
\end{prop}

\subsection{Computing $H^n_{\mathrm{Nis}}(X,\mathbf{GW}_n^{n-1}(\L))$}
\label{ss:computinghngw}
While Proposition \ref{prop:surjectionfromchowmod2} is strong enough to imply the vanishing statement that appears in the course of our verification of Theorem \ref{thmintro:murthysplittingdimension4}, it is not hard to identify the connecting homomorphism $\partial$ explicitly.  The remainder of the section is devoted to this task; the discussion of this section will not be used explicitly elsewhere in the paper.

Recall that one can define Steenrod operations $Sq^2:Ch^n(X)\to Ch^{n+1}(X)$ (\cite[\S 8]{Brosnan03} or \cite{VRed}). If $\L$ is an invertible $\O_X$-module, write $c_1(\L)$ for its first Chern class in $Ch^1(X)$, and define twisted Steenrod operations
\[
Sq^2_{\L}:Ch^n(X) \longrightarrow Ch^{n+1}(X)
\]
by $Sq^2_{\L}(\alpha)=Sq^2(\alpha)+c_1(\L)\cdot \alpha$.

\begin{thm}\label{thm:cohomologyofgwandsteenrodsquares}
Let $X$ be a smooth scheme over a field $k$ such that $\mathrm{char}(k)\neq 2$. For any $n\in\N$, there is an exact sequence of the form
\[
Ch^{n-1}(X)\stackrel{Sq^2_{\L}}\longrightarrow Ch^n(X) \longrightarrow H^n_{\mathrm{Nis}}(X,\mathbf{GW}_n^{n-1}(\L)) \longrightarrow 0.
\]
\end{thm}

\begin{proof}
If $n=0$, then the result follows from \cite[Lemma 4.1]{Fasel08c}.  In case $n\geq 1$, we first observe that $Ch^{n-1}(X)$ is generated by the classes of integral subvarieties $Y$ of codimension $n-1$ in $X$. If $C$ is the normalization of $Y$, then we have a finite morphism $i:C\to Y\subset X$. If necessary, by removing points of codimension $\geq 2$ in $C$ and $n+1$ in $X$ (which, by inspection of the Gersten resolution, will not change the cohomology groups in which we are interested), we can assume that $C$ is smooth.  The rest of the proof is identical to that of \cite[Theorem 4.17]{AsokFaselThreefolds}.
\end{proof}

\section{Some homotopy sheaves of classical symmetric spaces}
\label{section:metastablecomputations}
In this section, we study the $\aone$-homotopy sheaves of $GL_{2n}/Sp_{2n}$.  In Proposition \ref{prop:connectivityofSL2nSp2ntoSLSp}, we establish a ``stable range" for $\bpi_i^{\aone}(GL_{2n}/Sp_{2n})$ by analyzing the connectivity of the stabilization morphism $GL_{2n}/Sp_{2n} \to GL/Sp$ (see Subsection \ref{ss:geometricrepresentability} for recollections about the latter space).  In Theorem \ref{thm:main}, we give a computation of the first non-stable $\aone$-homotopy sheaf of $GL_{2n}/Sp_{2n}$.  This computation yields, in particular, the computation of $\bpi_3^{\aone}({\mathbb A}^3 \setminus 0)$ mentioned in Theorem \ref{thmintro:pi3a3minus0}.  However, the more general computation has other applications, e.g., to obstructions to existence of algebraic symplectic structures on smooth varieties.


\subsection{The sheaf of connected components of $GL_{2n}/Sp_{2n}$}
\label{ss:connectedcomponents}
We first study the $\aone$-connected components of $GL_{2n}/Sp_{2n}$.

\begin{lem}
\label{lem:pfaffianone}
The inclusions $Sp_{2n} \hookrightarrow SL_{2n} \longrightarrow GL_{2n}$ yield a split $\aone$-fiber sequence of the form
\[
SL_{2n}/Sp_{2n} \longrightarrow GL_{2n}/Sp_{2n} {\longrightarrow} \gm{};
\]
a splitting is given by $t \mapsto diag(t,1\ldots,1)$. In particular, the first morphism is the inclusion of the $\aone$-connected component of the base-point.
\end{lem}

\begin{proof}
Since $GL_{2n}/Sp_{2n}=GL_{2n}\times^{SL_{2n}} SL_{2n}/Sp_{2n}$ the sequence
\[
SL_{2n}/Sp_{2n} \longrightarrow GL_{2n}/Sp_{2n} {\longrightarrow} \gm{}
\]
is a $\aone$-fibre sequence by \cite[Proposition 5.1]{Wendt}.

Next, to show that $SL_{2n}/Sp_{2n}$ is $\aone$-connected, it suffices by \cite[Lemma 6.1.3]{MStable} to show that for every finitely generated extension $L/k$ the simplicial set $Sing_*^{\aone}(SL_{2n}/Sp_{2n})(L)$ is connected.  Since the morphism $SL_{2n} \to SL_{2n}/Sp_{2n}$ is an epimorphism of Nisnevich sheaves, the connectedness of the simplicial set $Sing_*^{\aone}(SL_{2n}/Sp_{2n})(L)$ follows from the connectedness of the simpicial set $Sing_*^{\aone}(SL_{2n})(L)$, which is a consequence of the fact that any element of $SL_{2n}(L)$ can be written as a product of elementary matrices.

Since the morphism
\[
GL_{2n}/Sp_{2n} {\longrightarrow} \gm{}
\]
splits, it follows that $\bpi_0^{\aone}(GL_{2n}/Sp_{2n})=\bpi_0^{\aone}(\gm{})= \gm{}$, and this completes the verification of the lemma.
\end{proof}

\begin{notation}
Set $X_n:=SL_{2n}/Sp_{2n}$ and $X_{\infty}=SL/Sp$.
\end{notation}

\begin{cor}
\label{cor:glspconnectedcomponents}
There is a canonical isomorphism $\bpi_0^{\aone}(GL/Sp) \cong \gm{}$, and the induced morphism $X_\infty=SL/Sp \to GL/Sp$ is the inclusion of the $\aone$-connected component of the base-point.  In particular, there are canonical isomorphisms $\bpi_i^{\aone}(SL/Sp) \isomt \mathbf{GW}^3_{i+1}$ for any integer $i > 0$.
\end{cor}

\begin{proof}
Using Lemma \ref{lem:pfaffianone} and passing to the colimit with respect to $n$ one obtains an $\aone$-fiber sequence of the form
\[
SL/Sp \longrightarrow GL/Sp \longrightarrow \gm{}.
\]
Since $\gm{}$ is $\aone$-discrete, i.e., $\bpi_0^{\aone}(\gm{}) \cong \gm{}$ and higher $\aone$-homotopy sheaves of $\gm{}$ for any choice of base-point are trivial, and since $SL/Sp$ is $\aone$-connected, the result follows.
\end{proof}


The practical consequence of the above statements is that the $\aone$-homotopy theory of the map $GL_{2n}/Sp_{2n} \to GL/Sp$ is reduced to studying the map $X_n \to X_\infty$.

\subsection{Stabilization fiber sequences involving $X_n$}
\label{ss:stabilizationsequences}
Now, we study the map $X_n \to X_{n+1}$.  Following the development of ideas from classical topology, we will identify $X_n$ as the $\aone$-homotopy fiber of a morphism $X_{n+1} \to {\mathbb A}^{2n+1} \setminus 0$.  To accomplish this we first provide an alternative geometric model for $X_{n+1}$.  Begin by considering the commutative diagram
\begin{equation}
\label{eqn:spslfibersquare}
\xymatrix{
Sp_{2n} \ar[r]\ar[d] & Sp_{2n+2} \ar[d]\\
SL_{2n+1} \ar[r] & SL_{2n+2}
}
\end{equation}
where the top horizontal map is defined (functorially) by sending a symplectic matrix $M$ to the block-diagonal matrix $diag(M,Id_2)$, the bottom horizontal map sends a matrix $M$ of determinant $1$ to the block-diagonal matrix $diag(M,1)$, the left-hand map is determined by the fact that the composite $Sp_{2n}\to Sp_{2n+2}\to SL_{2n+2}$ factors through $SL_{2n+1}$ and the right-hand vertical map is the standard inclusion.  It is straightforward to check that this square is cartesian.

Since $Sp_{2n+2}$ acts transitively on $SL_{2n+2}/SL_{2n+1}$, and the stabilizer of the identity coset is $Sp_{2n}$, we conclude that there exists an isomorphism of schemes $SL_{2n+2}/SL_{2n+1} \cong Sp_{2n+2}/Sp_{2n}$. Analogously, we can deduce the following result, which provides our alternative model for $X_{n+1}$.

\begin{lem}
\label{lem:xmisomorphism}
For any integer $n \geq 1$, the map $SL_{2n+1}\to SL_{2n+2}$ induces an isomorphism of schemes $SL_{2n+1}/Sp_{2n} \cong X_{n+1}$.
\end{lem}

The following result provides a description of the connectivity of the $\aone$-homotopy fiber of the stabilization map $X_n \to X_{n+1}$, together with some complements.

\begin{prop}
\label{prop:connectivityofSL2nSp2ntoSLSp}
For any $n\geq 1$, there is an $\aone$-fiber sequence of the form
\[
X_n \longrightarrow X_{n+1} \longrightarrow SL_{2n+1}/SL_{2n}.
\]
Therefore, the map $X_n \to X_{n+1}$ is $(2n-2)$-$\aone$-connected, and the map $X_n \to SL/Sp$ is $(2n-2)$-$\aone$-connected.  Consequently, there are isomorphisms $\bpi_i^{\aone}(X_n) \isomt \mathbf{GW}^3_{i+1}$ for $i \leq 2n-2$ and there is an exact sequence of the form
\[
\bpi_{2n}^{\aone}(SL/Sp) \longrightarrow \K^{MW}_{2n+1} \longrightarrow \bpi_{2n-1}^{\aone}(X_n) \longrightarrow \bpi_{2n-1}^{\aone}(SL/Sp) \longrightarrow 0.
\]
\end{prop}

\begin{proof}
For the first statement, observe that we have a fiber sequence
\[
X_{n}=SL_{2n}/Sp_{2n} \longrightarrow SL_{2n+1}/Sp_{2n} \longrightarrow SL_{2n+1}/SL_{2n}
\]
by \cite[Proposition 5.2]{Wendt}. Now, Lemma \ref{lem:xmisomorphism} yields an isomorphism $X_{n+1}\cong SL_{2n+1}/Sp_{2n}$.  For the second statement, observe first that projecting a matrix to its last column defines a map $SL_{2n+1}/SL_{2n}\to \A^{2n+1}\setminus 0$ whose fibers are affine spaces. This map is thus an $\aone$-weak equivalence. Now $\A^{2n+1}\setminus 0$ is $(2n-1)$-$\aone$-connected by \cite[Theorem 6.38 and Remark 6.42]{MField}.  The third statement follows from the second by induction on $n$, and the fourth statement follows from the previous three together with Corollary \ref{cor:glspconnectedcomponents} by looking at the long exact sequence in homotopy sheaves attached to the stated $\aone$-fiber sequence and using the fact that $\bpi_{2n}^{\aone}(\A^{2n+1}\setminus 0) \cong \K^{MW}_{2n+1}$, i.e., \cite[Theorem 6.40]{MField}.
\end{proof}

\subsection{The first non-stable sheaf}
\label{ss:firstnonstable}
The goal of this section is to identify the cokernel of the morphism $\bpi_{2n}^{\aone}(SL/Sp) \longrightarrow \K^{MW}_{2n+1}$ appearing in Proposition \ref{prop:connectivityofSL2nSp2ntoSLSp}.  To perform this computation, we need to recall some notation and results from \cite[\S 3]{AsokFaselSpheres}; there we introduced strictly $\aone$-invariant sheaves $\mathbf{S}_m$ for any $m\geq 2$ and $\mathbf{T}_m$ for $m$ odd and $\geq 3$.

By definition, if $L/k$ is any finitely generated field extension, then ${\mathbf S}_m(L)$ is the cokernel of a homomorphism $\K_m^Q(L)\to \K_m^M(L)$ introduced by Suslin in \cite[\S 4]{Suslin82b}. There are surjective morphisms $\K_m^M/(m-1)!\to \mathbf{S}_m$ and, provided $m$ is odd, $\mathbf{S}_m\to \K_m^M/2$. The sheaf $\mathbf{T}_m$ for $m$ odd is then defined as the fiber product
\[
\xymatrix{
{\mathbf T}_m \ar[r]\ar[d] & \mathbf{I}^m \ar[d] \\
{\mathbf S}_m \ar[r] & \K^M_m/2,
}
\]
where $\mathbf{I}^m$ is the unramified sheaf corresponding to the $m$-th power of the fundamental ideal in the Witt ring.  See Subsection \ref{ss:strictlyaoneinvariant} (in particular the examples) for more details regarding strictly $\aone$-invariant sheaves.

\begin{thm}
\label{thm:main}
For $n\geq 1$, the canonical morphism $GL_{2n}/Sp_{2n} \to GL/Sp$ is $(2n-2)$-$\aone$-connected, and there is a short exact sequence of the form
\[
0 \longrightarrow {\mathbf F}_{2n+1} \longrightarrow \bpi_{2n-1}^{\aone}(GL_{2n}/Sp_{2n}) \longrightarrow \mathbf{GW}^3_{2n} \longrightarrow 0,
\]
where ${\mathbf F}_{2n+1}$ is a quotient of ${\mathbf T}_{2n+1}$.
\end{thm}

\begin{proof}
Recall first that $X_n$ can be identified as the $\aone$-connected component of the base point in $GL_{2n}/Sp_{2n}$. It follows that $\bpi_{2n-1}^{\aone}(X_n)=\bpi_{2n-1}^{\aone}(GL_{2n}/Sp_{2n})$.  By means of Proposition \ref{prop:geometricGW3}, the exact sequence in Proposition \ref{prop:connectivityofSL2nSp2ntoSLSp} takes the form
\[
\mathbf{GW}^3_{2n+1} \stackrel{\chi_{2n+1}}{\longrightarrow} \K^{MW}_{2n+1} \longrightarrow \bpi_{2n-1}^{\aone}(GL_{2n}/Sp_{2n}) \longrightarrow \mathbf{GW}^3_{2n} \longrightarrow 0,
\]
and we set
\[
{\mathbf F}_{2n+1} := \coker(\mathbf{GW}^3_{2n+1} \stackrel{\chi_{2n+1}}{\longrightarrow} \K^{MW}_{2n+1}).
\]
It remains to describe ${\mathbf F}_{2n+1}$ (here ${\mathbf F}$ stands for ``forgetful") as a quotient of ${\mathbf T}_{2n+1}$.

To this end, observe that the morphism $\chi_{2n+1}:\mathbf{GW}^3_{2n+1} \to \K^{MW}_{2n+1}$ is induced by the morphism $X_{n+1}\cong SL_{2n+1}/Sp_{2n}\to SL_{2n+1}/SL_{2n}$ by applying $\bpi^{\aone}_{2n}$.
There is a commutative triangle of the form:
\begin{equation}\label{eqn:triangle}
\xymatrix{\bpi^{\aone}_{2n}(SL_{2n+1})\ar[r]\ar[rd] & \bpi^{\aone}_{2n}(SL_{2n+1}/Sp_{2n})\ar[d] \\
 & \bpi^{\aone}_{2n}(SL_{2n+1}/SL_{2n}).}
\end{equation}
By \cite[Lemma 3.13]{AsokFaselSpheres}, the morphism $\bpi^{\aone}_{2n}(SL_{2n+1})\to \bpi^{\aone}_{2n}(SL_{2n+1}/SL_{2n})$ admits the following factorization
\[
\bpi^{\aone}_{2n}(SL_{2n+1})\longtwoheadrightarrow\K^Q_{2n+1} \longrightarrow 2\K^M_{2n+1} \subset \K^{MW}_{2n+1}
\]
where the first arrow is the stabilization morphism $\bpi^{\aone}_{2n}(SL_{2n+1})\to \bpi^{\aone}_{2n}(SL)$. Moreover, \cite[Lemma 3.10]{AsokFaselSpheres} shows that $\mathrm{Im}(\psi^\prime_{2n+1}:\K^Q_{2n+1}\to \K^{MW}_{2n+1})$ contains $(2n)! \K^M_{2n+1}$. On the other hand, the stabilization morphism
\[
\bpi^{\aone}_{2n}(SL_{2n+1}/Sp_{2n})\cong \bpi^{\aone}_{2n}(SL_{2n+2}/Sp_{2n+2}) \longrightarrow \bpi^{\aone}_{2n}(SL/Sp)
\]
is an isomorphism by Proposition \ref{prop:connectivityofSL2nSp2ntoSLSp}.  Now, the commutative square of spaces
\[
\xymatrix{SL_{2n+1}\ar[r]\ar[d] & SL_{2n+1}/Sp_{2n}\ar[d] \\
SL\ar[r] & SL/Sp}
\]
yields a commutative square of sheaves of the form:
\[
\xymatrix{\bpi^{\aone}_{2n}(SL_{2n+1})\ar[r]\ar@{->>}[d]  & \bpi^{\aone}_{2n}(SL_{2n+1}/Sp_{2n})\ar[d]^-{\cong} \\
\K_{2n+1}^Q\ar[r] & \mathbf{GW}_{2n+1}^3},
\]
where the bottom arrow is $H_{3,2n+1}$ by \cite[Theorem 8.4]{SchlichtingTripathi}. Since the left-hand vertical map is surjective, it follows that the commutative triangle (\ref{eqn:triangle}) induces a commutative triangle
\begin{equation}\label{eqn:triangle2}
\xymatrix@C=4em{\K_{2n+1}^Q\ar[r]^-{H_{3,2n+1}}\ar[rd]_-{\psi^\prime_{2n+1}} & \mathbf{GW}_{2n+1}^3\ar[d]^-{\chi_{2n+1}} \\
 & \K_{2n+1}^{MW}.}
\end{equation}
Now ${\mathbf T}_{2n+1} = \coker(\psi_{2n+1}^\prime)$ by  \cite[p. 18]{AsokFaselSpheres}, and we obtain an epimorphism ${\mathbf T}_{2n+1} \to {\mathbf F}_{2n+1}$ as claimed.
\end{proof}

\begin{rem}
Given an oriented vector bundle $V$ of rank $2n$ on a smooth affine variety $Y$, the obstruction for $V$ to carry a non-degenerate symplectic form can be analyzed using the above results. Indeed, $V$ gives rise to a map $Y\to BSL_{2n}$ and constructing a non-degenerate symplectic form on $V$ is equivalent to lifting this map to a map $Y\to BSp_{2n}$. The obstruction to the existence of such a lifting is governed by the homotopy fiber of $BSp_{2n}\to BSL_{2n}$ which is precisely $X_n$. The above computation of $\bpi^{\aone}_i(X_n)$ provides sufficient information to apply the machinery of obstruction theory, as described in Section \ref{section:splittingproblem}, to study this problem when $d\leq 2n$.
\end{rem}

\subsection{Equivariance with respect to $\gm{}$-actions}
\label{ss:actionsonpi3a3minus0}
With the exception of $SL_{2n+1}$, all the spaces appearing in Diagram (\ref{eqn:spslfibersquare}) are endowed with the action of $\gm{}$ defined in Subsection \ref{ss:actionsongwsheaves}. It is easy to check that the conjugation by $\gamma_{2n+2,\lambda}^{-1}$ on $SL_{2n+2}$ preserves the subgroup $SL_{2n+1}$.  With respect to the induced action on $SL_{2n+1}$, all the maps in Diagram (\ref{eqn:spslfibersquare}) are $\gm{}$-equivariant, and the isomorphism $SL_{2n+2}/Sp_{2n+2} \cong SL_{2n+1}/Sp_{2n}$ of Lemma \ref{lem:xmisomorphism} is $\gm{}$-equivariant as well.

With respect to the $\gm{}$-actions just mentioned, it follows that the $\aone$-fiber sequence of Proposition \ref{prop:connectivityofSL2nSp2ntoSLSp} is $\gm{}$-equivariant.  Thus, the exact sequence of sheaves appearing in the same proposition is also an exact sequence of $\gm{}$-equivariant sheaves.  By definition, the action of $\gm{}$ on $\bpi_{2n-1}^{\aone}(SL/Sp) \cong \mathbf{GW}^3_{2n}$ is the conjugation action described in Subsection \ref{ss:actionsongwsheaves}.  By Proposition \ref{prop:multiplicativeGLSP}, the conjugation action coincides with the multiplicative action of $\gm{}$.

It remains only to identify the action of $\gm{}$ on $\bpi_{2n}^{\aone}(SL_{2n+1}/SL_{2n})$ that is induced by $\gm{}$-action on $SL_{2n+1}/SL_{2n}$ just considered.  For this, observe that the map $SL_{2n+1}/SL_{2n} \to {\mathbb A}^{2n+1} \setminus 0$ given by sending a matrix to its last column can be made $\gm{}$-equivariant by equipping ${\mathbb A}^{2n+1} \setminus 0$ with the $\gm{}$-action given by the following formula:
\[
(t,x_1,\ldots,x_{2n+1})\longmapsto (x_1,tx_2,x_3,tx_4,\ldots,tx_{2n},x_{2n+1}).
\]
Arguing as in \cite[Lemma 4.8]{AsokFaselThreefolds}, we see that the action of $\gm{}$ on $\K^{MW}_{2n+1}=\piaone_{2n}(\A^{2n+1}\setminus 0)$ is trivial if $n$ is even, while it is the one coming from multiplication by $\K_0^{MW}=\mathbf{GW}_0^0$ if $n$ is odd.  Putting everything together, the $\gm{}$-action on the sheaf ${\mathbf F}_{2n+1}$ of Theorem \ref{thm:main} is then simply the trivial action if $n$ is even and the multiplicative action if $n$ is odd.

Specializing this discussion to the case $n = 2$, we obtain the following result, which is the main computation of homotopy sheaves used in the remainder of the paper.

\begin{thm}
\label{thm:pi3a3minus0}
There is a short exact sequence of the form
\[
0 \longrightarrow {\mathbf F}_{5} \longrightarrow \bpi_{3}^{\aone}({\mathbb A}^3 \setminus 0) \longrightarrow \mathbf{GW}^3_4 \longrightarrow 0,
\]
where ${\mathbf F}_5$ is a quotient of ${\mathbf T}_5$.  If we view $\bpi_3^{\aone}({\mathbb A}^3 \setminus 0)$ as a $\gm{}$-sheaf with action induced by the $\gm{}$-action on ${\mathbb A}^3 \setminus 0$ given by the formula $t \cdot (x_1,x_2,x_3) =  (x_1,tx_2,x_3)$, equip $\mathbf{GW}^3_4$ with the multiplicative $\gm{}$-action and ${\mathbf F}_5$ with the trivial $\gm{}$-action, then the above exact sequence is $\gm{}$-equivariant.
\end{thm}


\section{The Hopf map $\nu$ and $\pi_3^{\aone}({\mathbb A}^3 \setminus 0)$}
\label{section:hopfmap}
In the previous section, we described the sheaf $\bpi_3^{\aone}({\mathbb A}^3 \setminus 0)$ as an extension of the Grothendieck-Witt sheaf $\mathbf{GW}^3_4$ by a sheaf we called ${\mathbf F}_5$.  The goal here is to provide a better understanding of the ``topological" origin of the sheaf ${\mathbf F}_5$.  While it is not explicit in the statement of Theorem \ref{thm:pi3a3minus0}, since the sheaf ${\mathbf F}_5$ admits an epimorphism from ${\mathbf T}_5$, ${\mathbf T}_5$ is a fiber product of $\mathbf{S}_5$ and $\mathbf{I}^5$ over $\K^M_5/2$, and $\mathbf{S}_5$ admits an epimorphism from $\K^M_5/24$, the number ``$24$" appears naturally in the description of ${\mathbf F}_5$.  We will see that the number $24$ appearing in the description of ${\mathbf F}_5$ is the ``same" as the $24$ the appears in the third stable homotopy group of spheres.  We place the word ``topological" in quotes because the initial computations we make are purely algebraic before studying what happens under real and complex realization.

\subsection{Structure of the contractions of $\mathbf{F}_5$}
\label{ss:contractionsoff5}
The computation of Theorem \ref{thm:pi3a3minus0} can be refined to provide more detailed information about ${\mathbf F}_5$: the next result shows that the epimorphism ${\mathbf T}_5 \to {\mathbf F}_5$ becomes an isomorphism after repeated contraction (see Subsections \ref{ss:contracted} and \ref{ss:contractions} for a discussion of contractions). At that point, we should point out that neither $\mathbf F_5$ nor $\mathbf T_5$ are known explicitly \emph{before} contracting them several times.

\begin{lem}\label{lem:contractedT5}
For any even integer $n\in\N$, the epimorphism ${\mathbf T}_{2n+1}\to {\mathbf F}_{2n+1}$ induces an isomorphism $({\mathbf T}_{2n+1})_{-2n}\to ({\mathbf F}_{2n+1})_{-2n}$ and there is a cartesian square of the form:
\[
\xymatrix{({\mathbf T}_{2n+1})_{-2n}\ar[r]\ar[d] & {\mathbf I}\ar[d] \\
\K^M_1/(2n)!\ar[r] & \K^M_1/2; }
\]
here the lower horizontal map factors the canonical quotient morphism $\K^M_1 \to \K^M_1/2$.
\end{lem}

\begin{proof}
Contracting Diagram (\ref{eqn:triangle2}) $2n$-times and using Lemmas \ref{lem:contractionofgw} and \ref{lem:karoubiperiodicitycontractions}, we obtain a commutative triangle of the form
\[
\xymatrix@C=4em{\K_{1}^Q\ar[r]^-{H_{3,1}}\ar[rd]_-{(\psi^\prime_{2n+1})_{-2n}} & \mathbf{GW}_{1}^3\ar[d]^-{(\chi_{2n+1})_{-2n}} \\
 & \K_{1}^{MW}.}
\]
Since $(\mathbf F_{2n+1})_{-2n}=\coker ((\chi_{2n+1})_{-2n})$ and $(\mathbf T_{2n+1})_{-2n}=\coker ((\psi^\prime_{2n+1})_{-2n})$, it suffices to show that $H_{3,1}$ is surjective to prove our claim. Since all sheaves involved are strictly $\aone$-invariant, it suffices to check this on sections over fields (see Subsection \ref{ss:strictlyaoneinvariant}).  In fact, by \cite[Lemma 2.3]{FaselRaoSwan}, the induced map on sections over fields is an isomorphism.  The fact that $(\mathbf{T}_{2n+1})_{-2n}$ sits in a Cartesian square of the stated form is then a direct consequence of \cite[Corollary 3.11]{AsokFaselSpheres}.
\end{proof}

\begin{rem}
It seems that the epimorphism $\mathbf{T}_{2n+1} \to \mathbf{F}_{2n+1}$ does not, in general, become an isomorphism after $(2n-1)$-fold contraction.  Indeed, consider the morphism $\mathbf T_5\to \mathbf F_5$.  If we contract this morphism $3$ times, the map $H_{0,2}:\K_2^Q\to \mathbf{GW}_2^0$ has the constant sheaf $\pm 1$ as cokernel and it is quite possible that it contributes to the image of $(\chi_5)_{-3}$. See also the next remark.
\end{rem}



\begin{rem}\label{rem:comparisonft}
In classical topology, the first few non-stable homotopy groups of $X_{n}(\cplx)$ were computed in \cite{BHarris2}.  In particular, if $n$ is even, then $\pi_{4n}(X_n(\cplx)) = \Z/(2n!)$, while if $n$ is odd, then $\pi_{4n}(X_n(\cplx)) = \Z/((2n!)/2)$.  Complex realization gives a morphism
\[
\bpi_{2n-1,2n+1}^{\aone}(X_n) \longrightarrow \pi_{4n}(X_n(\cplx)).
\]
The group $\bpi_{2n-1,2n+1}^{\aone}(X_n)$ can be computed by contracting the result of Theorem \ref{thm:main} $(2n+1)$ times.  Since $(\mathbf{GW}^3_{2n})_{-2n-1}=\mathbf{W}^{1-2n}$ is trivial by \cite[Proposition 5.2]{Balmer02}, it follows that $\bpi_{2n-1,2n+1}^{\aone}(X_n) = ({\mathbf F}_{2n+1})_{-2n-1}$, and furthermore there is an epimorphism from $({\mathbf T}_{2n+1})_{-2n-1}$ onto this group. The latter contraction was discussed in detail in \cite[Corollary 3.11]{AsokFaselSpheres}, where it was established that $({\mathbf T}_{2n+1})_{-2n-1} = \Z/(2n!)$.  The classical computation suggests that, when $n$ is even, the homomorphism ${\mathbf T}_{2n+1} \to {\mathbf F}_{2n+1}$ is an isomorphism, while if $n$ is odd then it has a non-trivial kernel.
\end{rem}

\subsection{$GW$-module structures}
\label{ss:gwmodulestructure}
By \cite[Corollary 6.43]{MField}, we know that there is a ring isomorphism $\bpi_{3,5}^{\aone}(\Sigma^3_s \gm{\sma 5})(k) \cong \K^{MW}_0(k)$ (the ring structure comes from composition).  Therefore, precomposing with elements of $\bpi_{3,5}^{\aone}(\Sigma^3_s \gm{\sma 5})(k)$ gives $\bpi_{3,5}^{\aone}({\mathscr X})(k)$ a $\K^{MW}_0(k)$-module structure for any pointed space $\mathscr{X}$, and this module structure is covariantly functorial in $\mathscr{X}$ by construction.  In particular $\bpi_{3,5}^{\aone}({\mathbb A}^3 \setminus 0)(k)$ admits the structure of a $\K^{MW}_0(k)$-module.

The pair $(\Sigma_s^1,\Omega_s^1)$ given by simplicial suspension and simplicial looping form an adjoint pair, and the unit map of this adjunction together with the $\aone$-weak equivalence $\Sigma^4_s \gm{\sma 5} \cong SL_5/SL_4$ yield a composite map:
\[
\Sigma^3_s \gm{\sma 5} \longrightarrow \Omega^1_s \Sigma^4_s \gm{\sma 5} \cong \Omega^1_s SL_5/SL_4.
\]
We saw above that the connecting morphism in the $\aone$-fiber sequence $X_2 \to X_3 \to SL_5/SL_4$ is a morphism $\delta: \Omega^1_s SL_5/SL_4 \to X_2=SL_4/Sp_4 \cong {\mathbb A}^3 \setminus 0$.  Abusing notation, we write
\[
\delta: \Sigma^3_s \gm{\sma 5} \longrightarrow \Omega^1_s SL_5/SL_4 \longrightarrow X_2
\]
for the composite map.

The space $\Sigma^3_s \gm{\sma 5}$ being $2$-$\aone$-connected by \cite[Corollary 6.43]{MField}, Morel's $\aone$-Freudenthal suspension theorem \cite[Theorem 6.61]{MField} shows that the unit map in the previous paragraph induces an isomorphism upon applying $\bpi_{3,j}$ for any $j \geq 0$. In particular, the morphism
\[
\K^{MW}_0 = \bpi_{3,5}^{\aone}(\Sigma^3_s \gm{\sma 5}) \longrightarrow \bpi_{3,5}^{\aone}(\Omega^1_s SL_5/SL_4)
\]
is an isomorphism.  Using this notation, we can now describe the $\K^{MW}_0(k)$-module structure of $\bpi_{3,5}^{\aone}({\mathbb A}^3 \setminus 0)(k)$.

\begin{prop}
\label{prop:gwmodulestructureofpi35a3minus0}
There is a canonical isomorphism
\[
\bpi_{3,5}^{\aone}({\mathbb A}^3 \setminus 0) \cong \Z/24 \times_{\Z/2} {\mathbf W},
\]
and $\bpi_{3,5}^{\aone}({\mathbb A}^3 \setminus 0)(k)$ is generated as a $\K^{MW}_0(k)$-module by $\delta$.
\end{prop}

\begin{proof}
By Theorem \ref{thm:pi3a3minus0}, we know that $\bpi_3^{\aone}({\mathbb A}^3 \setminus 0)$ is an extension of $\mathbf{GW}^3_4$ by ${\mathbf F}_5$.  By \cite[Theorem 6.13]{MField}, we know that $\bpi_{3,5}^{\aone}({\mathbb A}^3 \setminus 0) \cong \bpi^{\aone}_{3}({\mathbb A}^3 \setminus 0)_{-5}$.  Since $\bpi_{3,5}^{\aone}(X_3) = (\mathbf{GW}^3_4)_{-5} = 0$ by Propositions \ref{prop:connectivityofSL2nSp2ntoSLSp} and \ref{prop:vanishingofcontractions}, it follows from the long exact sequence in $\aone$-homotopy sheaves associated with the $\aone$-fiber sequence $X_2 \to X_3 \to SL_5/SL_4$ and the identification $X_2\simeq {\mathbb A}^3 \setminus 0$ that the morphism
\[
\K^{MW}_0 = \bpi_{3,5}^{\aone}(\Omega^1_s SL_5/SL_4) \longrightarrow \bpi_{3,5}^{\aone}({\mathbb A}^3 \setminus 0)
\]
is an epimorphism. In other words, $\bpi_{3,5}^{\aone}({\mathbb A}^3 \setminus 0)(k)$ is generated as a $\K^{MW}_0(k)$-module by the connecting homomorphism $\delta$.

By exactness of contractions, and the fact that $(\mathbf{GW}^3_{4})_{-5} = 0$, it follows that
\[
\bpi_{3,5}^{\aone}({\mathbb A}^3 \setminus 0) \cong ({\mathbf F}_5)_{-5}.
\]
The result follows from Lemma \ref{lem:contractedT5}.
\end{proof}

\subsection{Complex realization}
\label{ss:complexrealization}
If $k = \cplx$, we observe that $X_2(\cplx)\cong (\A^3\setminus 0)(\cplx)\cong S^5$ and that $(SL_5/SL_4)(\cplx)\cong (\A^5\setminus 0)(\cplx)\cong S^9$. On the other hand, $X_3(\cplx)\cong (SL_6/Sp_6)(\cplx)\cong SU(6)/Sp(6)$. We can then apply the complex realization functor to the $\aone$-fiber sequence
\[
X_2 \longrightarrow X_3 \longrightarrow SL_5/SL_4.
\]
While complex realization does not in general send $\aone$-fiber sequences to fiber sequences, the sequence of spaces obtained by taking complex points happens to be a topological fiber sequence.  Thus, after shifting, we obtain a topological fiber sequence of the form
\[
\Omega^1 S^9 \longrightarrow S^5 \longrightarrow SU(6)/Sp(6).
\]
By precomposing the map $\Omega^1 S^9 \to S^5$ with the suspension map $S^8 \to \Omega^1 \Sigma^1 S^8$, we obtain a map $S^8 \to S^5$, which we want to identify.  The long exact sequence in homotopy groups of the above fiber sequence yields:
\[
\pi_8(\Omega^1 S^9) \longrightarrow \pi_8(S^5) \longrightarrow \pi_8(SU(6)/Sp(6)) \longrightarrow 0.
\]
We know that $\pi_8(SU(6)/Sp(6)) \cong \pi_8(SU/Sp) = \pi_8(U/Sp)$, and by Bott periodicity, we know that $\pi_8(U/Sp) = \pi_{10}(O) = \pi_2(O) = 0$.  In other words, the portion of the long exact sequence displayed above collapses to the surjection
\[
\Z \longrightarrow \pi_8(S^5) \longrightarrow 0.
\]
One knows that $\pi_8(S^5)$ is $\Z/24$ and it follows from the above that the map $S^8 \to S^5$ can be taken to be a generator.  On the other hand, $\pi_8(S^5)$ has a classical generator $\Sigma\nu^{top}$ where $\nu^{top}$ is the topological Hopf map.  It follows that, up to multiplication by a unit in $\Z/24$, the generator we constructed coincides with $\Sigma \nu^{top}$.

\begin{cor}
\label{cor:complexrealization}
Under complex realization, the homomorphism
\[
\bpi_{3,5}^{\aone}({\mathbb A}^3 \setminus 0)(\cplx) \longrightarrow \pi_8(S^5) = \Z/24
\]
is an isomorphism sending the class of the map $\delta$ to $u\Sigma \nu^{top}$, where $\nu^{top}: S^7 \to S^4$ is the topological Hopf map and $u \in \Z/24^{\times}$ is a unit.
\end{cor}

\begin{proof}
By Proposition \ref{prop:gwmodulestructureofpi35a3minus0}, we know that $\bpi_{3,5}^{\aone}({\mathbb A}^3 \setminus 0) \cong \Z/24 \times_{\Z/2} {\mathbf W}(k)$.  As a consequence, complex realization yields a homomorphism $\Z/24 \longrightarrow \Z/24$.  Moreover, we saw before the statement that the generator of (the topological) $\Z/24$ is precisely the connecting homomorphism in the fibration associated with the complex points of $X_2 \to X_3 \to {\mathbb A}^5 \setminus 0$.  Since this connecting homomorphism is algebraically defined, and $\delta$ is a generator of $\bpi_{3,5}^{\aone}({\mathbb A}^3 \setminus 0)(\cplx)$, it follows that complex realization maps the algebraic generator to the topological generator and is therefore an isomorphism.
\end{proof}

\begin{rem}
We interpret this result as saying that the $24$ that appears in $\bpi_3^{\aone}({\mathbb A}^3 \setminus 0)$ is the ``same" as that appearing in the third stable homotopy group of the spheres.
\end{rem}

\subsection{Real realization}
\label{ss:realrealization}
We can compute the homotopy groups of the real points as well to study real realization.  We view this computation as providing an explanation for the appearance of ${\mathbf W}$ in $\bpi_{3,5}^{\aone}({\mathbb A}^3 \setminus 0)$.  Up to $\aone$-homotopy, the fiber sequence $X_2 \to X_3 \to SL_5/SL_4$ yields the sequence
\[
{\mathbb A}^3 \setminus 0 \longrightarrow SL_6/Sp_6 \longrightarrow {\mathbb A}^5 \setminus 0,
\]
which, upon taking real points, gives the topological fiber sequence
\[
S^2 \longrightarrow SO(6)/U(3) \longrightarrow S^4.
\]
A computation using Bott periodicity shows that $\pi_3(SO(6)/U(3)) = \pi_3(O/U) = \pi_4(O) = 0$.  As a consequence, the map $\pi_3(\Omega^1 S^4) \to \pi_3(S^2)$ in the long exact sequence is an isomorphism.  Thus the composite map $S^3 \to \Omega^1 S^4 \to S^2$ is precisely the classical Hopf map $\eta$.

\begin{cor}
\label{cor:realrealization}
For any $j \geq 0$, real realization defines a surjective map $\bpi_{3,j}^{\aone}({\mathbb A}^3 \setminus 0)(\real) \to \pi_3(S^2) = \Z$.
\end{cor}

\begin{proof}
We can compute $\bpi_{3,j}^{\aone}({\mathbb A}^3 \setminus 0)$ by contracting $\bpi_3^{\aone}({\mathbb A}^3 \setminus 0)$ $j$-times.  To prove surjectivity, we will consider only ${\mathbf F}_5$, i.e., the kernel of the map $\bpi_3^{\aone}({\mathbb A}^3 \setminus 0) \to \mathbf{GW}^3_4$.  By definition, we know that ${\mathbf F}_5$ admits an epimorphism from ${\mathbf T}_5$, which is a fiber product of ${\mathbf S}_5$ and ${\mathbf I}^5$ over $\K^M_5/2$. Moreover, the map ${\mathbf T}_5\to {\mathbf F}_5$ is injective on ${\mathbf I}^5$. Contracting repeatedly and using the fact that ${\mathbf I}^i(\real) = \Z$ for any $i \leq 5$ (by convention ${\mathbf I}^i = {\mathbf W}$ for $i \leq 0$), we see that $\bpi_{3,j}^{\aone}({\mathbb A}^3 \setminus 0)(\real)$ is non-trivial.  The real realization of the connecting homomorphism lifts the generator of $\pi_3(S^2)$ by the discussion preceding the statement.
\end{proof}

\begin{rem}
\label{rem:i5isalowdimensionalaccident}
The above computation shows that the sheaf ${\mathbf I}^5$ appearing in the description of ${\mathbf F}_5$ is an avatar of the topological Hopf map $\eta: S^3 \to S^2$.  Since the topological Hopf map $\eta$ becomes $2$-torsion in $\pi_4(S^3)$, we expect that the factor of ${\mathbf I}^5$ appearing in $\bpi_3^{\aone}({\mathbb A}^3 \setminus 0)$ will become trivial after a single simplicial suspension, i.e., in $\bpi_4^{\aone}({\pone}^{\sma 3})$.  In particular, it should follow that $\bpi_{4,5}^{\aone}({\pone}^{\sma 3}) = \Z/24$.
\end{rem}

\section{Obstruction theory and the splitting problem}
\label{section:splittingproblem}
In this section, we explain in detail the obstruction theoretic computations required to reduce the splitting problem to the computation of $\aone$-homotopy sheaves.  We then explain how the computation of $\bpi_3^{\aone}({\mathbb A}^3 \setminus 0)$ yields the statement of the introduction.

By \cite[Theorem 8.1]{MField} we know that if $X$ is a smooth affine $k$-scheme, then $[X,BGL_n]_{\aone}$ is canonically in bijection with the set of isomorphism classes of rank $n$ vector bundles on $X$.  Consider the morphism $i_n:GL_n \to GL_{n+1}$ defined by mapping an invertible $n \times n$-matrix $M$ to the block-diagonal matrix $diag(M,1)$, and the induced morphism $BGL_n \to BGL_{n+1}$.  We study the Moore-Postnikov factorization of this map.  In the context of $\aone$-homotopy theory, the Moore-Postnikov factorization is discussed in \cite[Appendix B]{MField}.  For the convenience of the reader, Subsection \ref{ss:moorepostnikov} contains the precise form of the tower we use, and Subsection \ref{ss:specializingtomurthy} provides details of the construction in the special case we consider. The remainder of the section is devoted to explicit cohomological identifications of the obstructions that appear in this setting.

\subsection{Moore-Postnikov factorizations in $\aone$-homotopy theory}
\label{ss:moorepostnikov}
If $f:(\mathscr{E},e)\to (\mathscr{B},b)$ is a pointed morphism of spaces, we can, roughly speaking, apply the Postnikov tower construction to the $\aone$-homotopy fiber of $f$ to obtain an inductively defined sequence of obstructions to lifting.  In the setting of simplicial homotopy theory, the Moore-Postnikov factorization of a map is described in \cite[Chapter VI]{GoerssJardine}.  In the situation we consider, the space $\mathscr{B}$ will have a non-trivial $\aone$-fundamental sheaf of groups and the $\aone$-fibration $f$ will have the property that the action of $\bpi_1^{\aone}(\mathscr{B},b)$ on the homotopy groups of the $\aone$-homotopy fiber is non-trivial.  A detailed analysis of Moore-Postnikov factorizations in this ``non-simple" case can also be found in \cite{Robinson}.  For the convenience of the reader, we review the relevant statements we need here.

If ${\mathscr G}$ is a (Nisnevich) sheaf of groups, and ${\mathbf A}$ is a (Nisnevich) sheaf of abelian groups on which $\mathscr{G}$ acts, there is an induced action of ${\mathscr G}$ on the Eilenberg-Mac Lane space $K({\mathbf A},n)$ that fixes the base-point.  In that case, we set $K^{\mathscr G}(\mathbf A,n) := E{\mathscr G} \times^{\mathscr G} K({\mathbf A},n)$ (see Subsection \ref{ss:classifyingspaces} and \cite[Appendix B.2]{MField}). Projection onto the first factor defines a morphism $K^{\mathscr G}({\mathbf A},n) \to B{\mathscr G}$ that is split by the inclusion of the base-point of $K({\mathbf A},n)$.  To simplify the notation, we will suppress base-points in $\aone$-homotopy sheaves in the sequel.

\begin{thm}
\label{thm:moorepostnikovtower}
Suppose $f:(\mathscr{E},e)\to (\mathscr{B},b)$ is a pointed map of $\aone$-connected spaces and let $\mathscr{F}$ be the $\aone$-homotopy fiber of $f$.  Assume, in addition, that $f$ is an $\aone$-fibration , $\mathscr{B}$ is $\aone$-local and that $\mathscr{F}$ is $\aone$-simply connected.  There are pointed spaces $(\mathscr{E}^{(i)},e_i)$, $i \in {\mathbb N}$, with $\mathscr{E}^{(0)} = \mathscr{B}$, pointed morphisms
\[
\begin{split}
g^{(i)}:& \mathscr{E} \longrightarrow \mathscr{E}^{(i)}, \\
h^{(i)}:& \mathscr{E}^{(i)} \longrightarrow \mathscr{B}, \text{ and } \\
p^{(i)}:& \mathscr{E}^{(i+1)} \longrightarrow \mathscr{E}^{(i)},
\end{split}
\]
and commutative diagrams
\[
\xymatrix@C=4em{  & \mathscr{E}^{(i+1)}\ar[d]_-{p^{(i)}}\ar[rd]^-{h^{(i+1)}} & \\
\mathscr{E}\ar[r]_-{g^{(i)}}\ar[ru]^-{g^{(i+1)}} & \mathscr{E}^{(i)}\ar[r]_-{h^{(i)}} & \mathscr{B}}
\]
having the following properties:
\begin{enumerate}[i)]
\item The composite $h^{(i)}g^{(i)} = f$ for any $i\geq 0$.
\item The morphisms $\piaone_n(\mathscr{E})\to \piaone_n (\mathscr{E}^{(i)})$ induced by $g^{(i)}$ are isomorphisms for $n\leq i$ and an epimorphism for $n = i+1$.
\item The morphisms $\piaone_n(\mathscr{E}^{(i)})\to \piaone_n(\mathscr{B})$ induced by $h^{(i)}$ are isomorphisms for $n>i+1$, and a monomorphism for $n = i+1$.
\item The induced map ${\mathscr E} \to \operatorname{holim}_i {\mathscr E}^{(i)}$ is an $\aone$-weak equivalence.
\end{enumerate}
The morphisms $p^{(i)}$ are $\aone$-fibrations with $\aone$-homotopy fiber $K(\piaone_i(\mathscr F),i)$ for any $i\geq 0$, and $p^{(i)}$ is a {\em twisted $\aone$-principal fibration}, i.e., there is a unique (up to $\aone$-homotopy) morphism
\[
k_{i+1}: \mathscr{E}^{(i)} \longrightarrow K^{\bpi_1^{\aone}(\mathscr{B})}(\bpi_{i}^{\aone}(\mathscr{F}),i+1)
\]
called a $k$-invariant sitting in an $\aone$-homotopy pullback square of the form
\[
\xymatrix{
\mathscr{E}^{(i+1)} \ar[r]\ar[d] & B\bpi_1^{\aone}(\mathscr{B}) \ar[d] \\
\mathscr{E}^{(i)} \ar[r]^-{k_{i+1}} & K^{\bpi_1^{\aone}(\mathscr{B})}(\bpi_{i}^{\aone}(\mathscr{F}),i+1),
}
\]
where the map $B\bpi_1^{\aone}(\mathscr{B})\to K^{\bpi_1^{\aone}(\mathscr{B})}(\bpi_{i}^{\aone}(\mathscr{F}),i+1)$ is given by inclusion of the base point.
\end{thm}

Given a smooth scheme $X$, if $X_+$ is $X$ with a disjoint base-point attached, we use the forgetful map to identify the set of pointed maps $X_+ \to \mathscr{B}$ with the set of unpointed maps $X \to \mathscr{B}$.  Given a map $f: X \to \mathscr{B}$, Theorem \ref{thm:moorepostnikovtower} gives an inductively defined sequence of obstructions to lifting $f$ to a map $\tilde{f}: X \to \mathscr{E}$.

Assume we are given a map $f^{(i)}: X \to \mathscr{E}^{(i)}$, the $\aone$-homotopy cartesian square of Theorem \ref{thm:moorepostnikovtower} shows $f^{(i)}$ lifts to a map $f^{(i+1)}: X \to \mathscr{E}^{(i+1)}$ if and only if the composite map
\[
X \longrightarrow \mathscr{E}^{(i)}\stackrel{k_{i+1}} \longrightarrow K^{\bpi_1^{\aone}(\mathscr{B})}(\bpi_{i}^{\aone}(\mathscr{F}),i+1)
\]
lifts through a map $X\to B\bpi_1^{\aone}(\mathscr{B})$.

The composite map $X \to K^{\bpi_1^{\aone}(\mathscr{B})}(\bpi_{i}^{\aone}(\mathscr{F}),i+1)$ admits a cohomological interpretation, which we now describe.  Composing this map with the projection map $K^{\bpi_1^{\aone}(\mathscr{B})}(\bpi_{i}^{\aone}(\mathscr{F}),i+1)\to B\bpi_1^{\aone}(\mathscr{B})$, we obtain a map $X\to B\bpi_1^{\aone}(\mathscr{B})$, which, in turn, yields a $\bpi_1^{\aone}(\mathscr{B})$-torsor $\mathcal P$.  The action of $\bpi_1^{\aone}(\mathscr{B})$ on $\bpi_{i}^{\aone}(\mathscr{F})$ is then used to equip $\bpi_{i}^{\aone}(\mathscr{F})$ with the structure of a $\Z[\bpi_1^{\aone}(\mathscr{B})]$-module.  We can then consider the twisted sheaf:
\[
\bpi_{i}^{\aone}(\mathscr{F})(\mathcal P) := \Z[{\mathcal P}] \tensor_{\Z[{\piaone_1(\mathscr{B})}]} \bpi_{i}^{\aone}(\mathscr{F})
\]
on the small Nisnevich site of $X$.  The map $X\to K^{\bpi_1^{\aone}(\mathscr{B})}(\bpi_{i}^{\aone}(\mathscr{F}),i+1)$ can thus be seen as a $\piaone_1(\mathscr{B})$-equivariant map
\[
{\mathcal P} \longrightarrow K(\bpi_{i}^{\aone}({\mathscr F}),i+1)\times {\mathcal P},
\]
which is the identity on the second factor.  As explained in \cite[\S 6]{AsokFaselThreefolds}, using \cite[Lemma B.15]{MField}, we use the above to identify the obstruction with an element of the group $H^{i+1}_{\Nis}(X,\bpi_{i}^{\aone}(\mathscr{F})(\mathcal P))$.

If the resulting element of this group is trivial, we may choose a lift to the next stage of the Moore-Postnikov tower.  As explained in \cite[\S 6]{AsokFaselThreefolds}, the set of lifts can be described as a quotient of the cohomology group $H^i_{\Nis}(X,\bpi_{i}^{\aone}(\mathscr{F})(\mathcal P))$.  If $X$ is a smooth scheme, then $X$ has finite Nisnevich cohomological dimension coinciding with its Krull dimension $dim X$.  As a consequence, if we can lift $X \to \mathscr{B}$ to a map $X \to \mathscr{E}^{(i)}$ for some $i > dim X$, the next obstruction to lifting necessarily vanishes, and the subsequent choice of lift is uniquely specified, i.e., there is a unique lift of the given map to a map $X \to \mathscr{E}^{(i+1)}$, and the map
\[
[X,\mathscr{E}^{(i)}]_{\aone} \longrightarrow [X,\mathscr{E}^{(i+1)}]_{\aone}
\]
is a bijection. Combining this with point (iv) of Theorem \ref{thm:moorepostnikovtower} allows us to conclude that $[X,\mathscr{E}]_{\aone} \to [X,\mathscr{E}^{(i))}]_{\aone}$ is a bijection for any $i > \dim X$.


\subsection{Specializing to ${\mathbb A}^{n+1} \setminus 0 \to BGL_{n} \to BGL_{n+1}$}
\label{ss:specializingtomurthy}
We now specialize the discussion of Theorem \ref{thm:moorepostnikovtower} to the situation mentioned in the introduction, i.e., consider the morphism $BGL_n \to BGL_{n+1}$ induced by the homomorphism of algebraic groups $GL_n \to GL_{n+1}$ sending an invertible matrix $M$ to the block-diagonal matrix $\operatorname{diag}(M,1)$.  By \cite[Lemma 3.10]{AsokPi1} the determinant homomorphism yields an isomorphism $\bpi_1^{\aone}(BGL_n) \isomt \gm{}$.  Since $BSL_n$ is $\aone$-$1$-connected by \cite[Theorems 6.50 and 7.20]{MField}, one knows that $BSL_n$ is $\aone$-weakly equivalent to a universal $\aone$-cover of $BGL_n$.  Therefore, there is an induced action of $\gm{}$ on $\bpi_i^{\aone}(BSL_n)$, and this action coincides with the action of $\gm{} \cong \bpi_1^{\aone}(BGL_n)$ on $\bpi_i^{\aone}(BGL_n)$ modulo identifications, which we now explicate.

Take $EGL_n/SL_n$ as a model for the $\aone$-universal cover of $GL_n$: indeed the natural map $EGL_n/SL_n \to EGL_n/GL_n$ is, by means of the isomorphism $GL_n/SL_n \cong \gm{}$, a $\gm{}$-torsor and, therefore, an $\aone$-covering space by \cite[Lemma 7.5]{MField}.  As the map $BSL_n = ESL_n/SL_n \to EGL_n/SL_n$ is a simplicial weak equivalence, it follows that $EGL_n/SL_n$ is an $\aone$-universal covering of $BGL_n$ by \cite[Theorem 7.8]{MField}.  Consider the splitting $\delta: \gm{} \to GL_n$ given by $t \mapsto (t,1,\ldots,1)$.  The action of $\gm{}$ on $EGL_n/SL_n$ can be modified by a homotopy using an argument similar to that in the proof of Proposition \ref{prop:multiplicativeGLSP}, as we now explain.

First, let us give a simplicial description of the space $EGL_n/SL_n$.  If $R$ is a smooth $k$-algebra, let $\mathscr{M}^{\circ}$ be the groupoid with one object $\bullet$ such that $Aut(\bullet) = SL_n(R)$, and let $\mathscr{M}$ be the groupoid whose objects are the cosets in $GL_n(R)/SL_n(R)$ and where $\hom_{\mathscr{M}}(A,B) = \{ M \in GL_n(R) | M \cdot A = B \}$.  The functor $\iota: \mathscr{M}^{\circ} \hookrightarrow \mathscr{M}$ sending $\bullet$ to the coset $Id_n \cdot SL_n$ provides a fully-faithful and essentially surjective functor and therefore yields an equivalence of categories.  The induced map $B\mathscr{M}^{\circ} \to B \mathscr{M}$ is, unwinding the definitions, and sheafifying, the simplicial equivalence $BSL_n \to EGL_n/SL_n$ mentioned in the previous paragraph.

Consider the functor $\Phi: \mathscr{M}^{\circ} \to \mathscr{M}^{\circ}$ given by $\Phi(M) = \delta(t)^{-1}M\delta(t)$.  Analogously, multiplication by $\delta(t)\cdot$ defines a functor $\mathscr{M} \to \mathscr{M}$. There is a natural transformation from the composite $\iota \circ \Phi$ to $(\delta(t) \cdot -) \circ \iota$, which, on the object $\bullet$, is given by $1 \cdot SL_n \mapsto \delta(t) \cdot SL_n$.  Taking nerves, we obtain a homotopy commutative square of the form
\[
\xymatrix{
BSL_n \ar[r]^{\Phi}\ar[d] & BSL_n \ar[d] \\
EGL_n/SL_n \ar[r] & EGL_n/SL_n
}
\]
where the bottom horizontal arrow is the action of $\gm{}$ on the $\aone$-universal cover of $BGL_n$ and the upper horizontal arrow is induced by conjugation by $\delta(t)^{-1}$ on $SL_n$.

Now, consider the morphism of $\aone$-fiber sequences
\[
\xymatrix{
SL_{n+1}/SL_n \ar[r]\ar[d]& BSL_n \ar[r]\ar[d] & BSL_{n+1} \ar[d] \\
GL_{n+1}/GL_{n} \ar[r] & BGL_n \ar[r] & BGL_{n+1},
}
\]
where the left vertical morphism is an isomorphism.  We are interested in the action of the sheaf $\gm{} \cong \bpi_1^{\aone}(BGL_{n+1})$ on $GL_{n+1}/GL_n$.  As the map $BGL_n \to BGL_{n+1}$ induces an isomorphism $\bpi_1^{\aone}(BGL_n) \cong \bpi_1^{\aone}(BGL_{n+1})$, the pullback of the universal cover of $BGL_{n+1}$ is also the universal cover of $BGL_n$.  Therefore, the action of $\gm{}$ on the universal cover of $BGL_n$ is the one induced by the action of $\gm{}$ on $BSL_n$ and $BSL_{n+1}$.  As the homotopy constructed in the previous paragraph is compatible with the inclusion $BSL_n \to BSL_{n+1}$, the induced homotopical action on $SL_{n+1}/SL_n$ coincides with the conjugation action.  The map $SL_{n+1}/SL_n \to {\mathbb A}^{n+1} \setminus 0$ induced by projection onto the last column is an $\aone$-weak-equivalence; if we equip the source with the conjugation action of $\gm{}$, and the target with the $\gm{}$-action defined by $t \cdot (x_1,\ldots,x_{n+1}) = (t^{-1}x_1,\ldots,x_{n+1})$, this $\aone$-weak equivalence is $\gm{}$-equivariant.  The formula just given for the action shows that the $\gm{}$-action on $\bpi_n^{\aone}({\mathbb A}^{n+1} \setminus 0) \cong \K^{MW}_{n+1}$ (again, use \cite[Theorem 6.40]{MField} for this identification) is the action described in Subsection \ref{ss:twists}.

Next, fix a vector bundle $\xi: \mathcal{E} \to X$ of rank $n+1$, and consider the classifying morphism $X \to BGL_{n+1}$; abusing notation, this classifying map will also be called $\xi$.  The $\bpi_1^{\aone}(BGL_{n+1}) = \gm{}$-torsor appearing in Subsection \ref{ss:moorepostnikov} is, by construction, the composite map $X \to BGL_{n+1} \to BGL_{n+1}^{(1)} = B\gm{}$ and this map classifies the $\gm{}$-torsor underlying the determinant line bundle $\det \xi$.  It follows that, assuming we already have a lift to the $i$th stage of the Moore-Postnikov tower of $BGL_{n} \to BGL_{n+1}$, the next obstruction to lifting $\xi$ to a map $X \to BGL_{n}$ is an element
\[
o_{i,n+1}(\xi) \in H^{i+1}_{\Nis}(X,\bpi_i^{\aone}({\mathbb A}^{n+1} \setminus 0)(\det \xi)).
\]
By \cite[Corollary 6.43]{MField}, the space ${\mathbb A}^{n+1} \setminus 0$ is $(n-1)$-$\aone$-connected and the first potentially non-trivial obstruction to lifting is $o_{n,n+1}(\xi)$; we refer to this obstruction as the primary obstruction to lifting.

Assume $X$ has dimension $n+2$.  In that case, $H^{i}_{\Nis}(X,\mathbf{A})$ vanishes for $i > n+2$ because the Nisnevich cohomological dimension of a smooth scheme is exactly its Krull dimension.  Therefore, the groups housing the obstructions all vanish for $i > n+1$.  In other words, if the primary obstruction to lifting vanishes, there is only $1$ additional potentially non-vanishing obstruction, i.e., $o_{n+1,n+1}$; we refer to this obstruction as the secondary obstruction to lifting, but it is only well-defined up to choice of a suitable lift to the $(n+1)$st-stage of the Moore-Postnikov factorization of our map.  If these two obstructions vanish, it follows by the discussion at the end of Subsection \ref{ss:moorepostnikov} that there exists a lift $X \to BGL_n$ of our initial map $\xi$, i.e., $\xi$ splits of a trivial rank $1$ bundle.

\subsection{The primary obstruction in Murthy's splitting conjecture}
\label{ss:primaryobstruction}
As mentioned above, the $\gm{}$-action on $\bpi_n^{\aone}({\mathbb A}^{n+1} \setminus 0) \cong \K^{MW}_{n+1}$ is the standard action described in Subsection \ref{ss:twists}, so $o_{n,n+1}(\xi)$ lives in the group $H^{n+1}_{\Nis}(X,\K^{MW}_{n+1}(\det \xi))$.  Under the additional assumptions that $k$ is algebraically closed having characteristic unequal to $2$ and $X$ is affine, the next result provides a sufficient condition for the vanishing of the primary obstruction in terms of Chern classes of $\xi$.

\begin{prop}
\label{prop:eulerclassalgclosedisachernclass}
If $k$ is an algebraically closed field having characteristic unequal to $2$, $X$ is a smooth affine $k$-variety of dimension $(n+2)$, and $\xi: \mathcal{E} \to X$ is a rank $(n+1)$ vector bundle on $X$, then $o_{n,n+1}(\xi)$ vanishes if $c_{n+1}(\xi) = 0$.
\end{prop}

\begin{proof}
In \cite[Corollary 5.3]{AsokFaselThreefolds}, we showed that, under the stated hypotheses, the canonical morphism
\[
H^{n+1}_{\Nis}(X,\K^{MW}_{n+1}(\det\xi)) \longrightarrow H^{n+1}_{\Nis}(X,\K^M_{n+1}) = CH^{n+1}(X)
\]
is an isomorphism. It follows that the element $o_{n,n+1}$ is canonically determined by an element of $CH^{n+1}(X)$.

The obstruction class on $X$ is pulled-back from the universal class on $BGL_{n+1}$, induced by the identity map on $BGL_{n+1}$.  More precisely, write $\mathcal{E}_{n+1}$ for the universal rank $(n+1)$ vector bundle on $BGL_{n+1}$.  In \cite[Theorem 2.3]{Asok13}, we explained how, after identifying $BGL_{n+1}$ with $Gr(n+1,\infty)$, to define the group $H^{n+1}_{\Nis}(BGL_{n+1},\K^{MW}_{n+1}(\det {\mathcal E}_{n+1}))$ using Totaro-style finite-dimensional approximations to $Gr(n+1,\infty)$ \cite[Definition 1.2]{Totaro99}.  Then, there is a commutative diagram of the form
\[
\xymatrix{
H^{n+1}_{\Nis}(BGL_{n+1},\K^{MW}_{n+1}(\det {\mathcal E}_{n+1})) \ar[r]\ar[d] & H^{n+1}_{\Nis}(X,\K^{MW}_{n+1}(\det \xi)) \ar[d] \\
H^{n+1}_{\Nis}(BGL_{n+1},\K^M_{n+1}) \ar[r] & H^{n+1}_{\Nis}(X,\K^M_{n+1})
}
\]

Since $o_{n,n+1}(\xi)$ is uniquely determined by its image in $H^{n+1}_{\Nis}(X,\K^M_{n+1})$, it suffices to understand the image $o_{n,n+1}$ of $o_{n,n+1}({\mathcal E}_{n+1})$ in $H^{n+1}(BGL_{n+1},\K^M_{n+1})$. Since we identified $BGL_{n+1}$ with the infinite Grassmannian $Gr_{n+1}$, it follows that the image of $o_{n,n+1}({\mathcal E}_{n+1})$ in the group $H^{n+1}(BGL_{n+1},\K^M_{n+1})$ is given by an element of $H^{n+1}(Gr_{n+1},\K^M_{n+1}) \cong CH^{n+1}(Gr_{n+1}) \cong \Z$ \cite[\S 15]{Totaro99}.  It follows that $o_{n,n+1}$ is a multiple of the $n+1$-st Chern class $c_{n+1}({\mathcal E}_{n+1})$, which is a generator of $CH^{n+1}(Gr_{n+1})$. Thus the image of $o_{n,n+1}(\xi)$ in $H^{n+1}(X,\K^M_{n+1})$ is an integer multiple of $c_{n+1}(\xi)$ and the result follows.
\end{proof}

\begin{rem}
\label{rem:nonalgebraicallyclosedsplitting}
The primary obstruction $o_{n,n+1}(\xi)$ is precisely the Euler class of $\xi$ as defined in \cite[Remark 8.15]{MField}. That this class coincides with the top Chern class under the projection map $H^{n+1}_{\Nis}(X,\K^{MW}_{n+1}(\det\xi)) \longrightarrow H^{n+1}_{\Nis}(X,\K^M_{n+1}) = CH^{n+1}(X)$ is a consequence of \cite[Theorem 1]{Asok13} (at least when $\xi$ is oriented), but we don't need this fact here.
\end{rem}

As a corollary, we see that Murthy's splitting conjecture follows from a simple cohomological vanishing statement.

\begin{cor}
\label{cor:murthysconjectureisvanishing}
If $X$ is a smooth affine scheme of dimension $d+1$ over an algebraically closed field having characteristic unequal to $2$, then Murthy's splitting conjecture holds for $X$ if, for any line bundle $\L$ on $X$, the cohomology group $H^{d+1}_{\Nis}(X,\bpi_d^{\aone}({\mathbb A}^d \setminus 0)(\L))$ is trivial.
\end{cor}

\subsection{Murthy's splitting conjecture for smooth affine $4$-folds}
\label{ss:secondaryobstruction}
In this section, we will establish Theorem \ref{thmintro:murthysplittingdimension4} from the introduction.  We described $\bpi_3^{\aone}({\mathbb A}^3 \setminus 0)$ together with a $\gm{}$-action on this sheaf in Theorem \ref{thm:pi3a3minus0}.  The discussion of Subsection \ref{ss:specializingtomurthy} shows that the $\gm{}$-action coming from the $\bpi_1^{\aone}(BGL_3)$-action on ${\mathbb A}^3 \setminus 0$ is precisely the action described in the theorem.

\begin{proof}[Proof of Theorem \ref{thmintro:murthysplittingdimension4}]
Theorem \ref{thm:pi3a3minus0} shows that there is a $\gm{}$-equivariant exact sequence of sheaves of the form
\[
0 \longrightarrow {\mathbf F}_5 \longrightarrow \piaone_3(\A^3\setminus 0) \longrightarrow \mathbf{GW}^3_4 \longrightarrow 0
\]
and thus for any line bundle $\L$ over $X$ an exact sequence of sheaves
\[
0 \longrightarrow {\mathbf F}_5 \longrightarrow \piaone_3(\A^3\setminus 0)(\L) \longrightarrow \mathbf{GW}^3_4(\L) \longrightarrow 0.
\]
By Corollary \ref{cor:murthysconjectureisvanishing} to prove Murthy's splitting conjecture in dimension $4$, it suffices to prove that $H^4_{\Nis}(X,\bpi_3^{\aone}({\mathbb A}^3 \setminus 0)(\L))$ vanishes for an arbitrary line bundle $\L$ on $X$.

The above exact sequence of sheaves yields an exact sequence of cohomology groups
\[
H^4_{\Nis}(X,{\mathbf F}_5) \longrightarrow H^4_{\Nis}(X,\piaone_3(\A^3\setminus 0)(\L)) \longrightarrow H^4_{\Nis}(X,\mathbf{GW}^3_4(\L)) \longrightarrow 0
\]
and it suffices to prove that both $H^4_{\Nis}(X,{\mathbf F}_5)$ and $H^4_{\Nis}(X,\mathbf{GW}^3_4(\L))$ are trivial to conclude.

We know that $\mathbf{F}_5$ admits an epimorphism from ${\mathbf T}_5$, which is itself a fiber product of ${\mathbf I}^5$ and a quotient of $\K^M_5/24$. Since $k$ is algebraically closed, it follows from \cite[Proposition 5.1]{AsokFaselThreefolds} that $H^4_{\Nis}(X,{\mathbf I}^5) = 0$. Likewise, it follows from \cite[Proposition 5.4]{AsokFaselThreefolds} that $H^4_{\Nis}(X,\K^M_5/24) = 0$.  Therefore, $H^4_{\Nis}(X,{\mathbf T}_5) = 0$.  Since $X$ has dimension $4$, the map $H^4_{\Nis}(X,{\mathbf T}_5) \to H^4_{\Nis}(X,{\mathbf F}_5)$ is surjective, so the latter group vanishes.

We are thus reduced to proving the vanishing of $H^4_{\Nis}(X,\mathbf{GW}^3_4(\L))$.  However, if $X$ is a smooth affine $4$-fold we know that $Ch^4(X)$ is trivial since $CH^4(X)$ is divisible (the argument for divisibility is standard: combine \cite[Example 1.6.6]{Fulton} and \cite[Proposition 1.8]{Fulton}).  Combining this observation with Proposition \ref{prop:surjectionfromchowmod2}, we conclude that if $X$ is a smooth affine $4$-fold over an algebraically closed field, then $H^4_{\Nis}(X,\mathbf{GW}^3_4(\L))$ vanishes.
\end{proof}

\begin{rem}
As the above discussion makes clear, the only obstruction to establishing Murthy's conjecture in general is a description of the sheaf $\bpi_{d}^{\aone}({\mathbb A}^d \setminus 0)$, or, in fact, simply $\bpi_{d,d}^{\aone}(\mathbb{A}^d \setminus 0)$ and $\bpi_{d,d+1}^{\aone}({\mathbb A}^d \setminus 0)$.  See \cite{AsokFaselOberwolfach} for a conjecture on the structure of $\bpi_d^{\aone}({\mathbb A}^d \setminus 0)$ for $d \geq 4$.
\end{rem}

\begin{footnotesize}
\bibliographystyle{alpha}
\bibliography{punctured}
\end{footnotesize}
\end{document}